\documentclass[10pt]{article}

\usepackage{amsmath,amsfonts,amssymb}
\usepackage{bbm}

\usepackage{lipsum}

\usepackage{xcolor}
\usepackage{ulem}

\newtheorem{defn}{Definition}[section]
\newtheorem{theorem}[defn]{Theorem}
\newtheorem{lemma}[defn]{Lemma}
\newtheorem{proposition}[defn]{Proposition}
\newtheorem{cor}[defn]{Corollary}
\newtheorem{remark}[defn]{Remark}
\newtheorem{example}[defn]{Example}

\newenvironment{proof}{{\bf Proof }}{{\vskip 0.1cm \hfill$\Box$}}

\def\N {{\mathbb N}}
\def\Q {{\mathbb Q}}
\def\R {{\mathbb R}}
\def\E{{\mathbb E}}
\def\P{{\mathbb P}}
\def\M{{\mathbb M}}
\newcommand{\F}{\mathcal{F}}

\begin{document}


\noindent
{\Large \bf {Pointwise well-posedness results for degenerate It\^{o}-SDEs with locally bounded drifts}}
\\ \\
\bigskip
\noindent
{\bf Haesung Lee, Gerald Trutnau{\footnote{The research of Gerald Trutnau was supported by the National Research Foundation of Korea (NRF) grant funded by the Korea government (MSIP) : NRF-2016K2A9A2A13003815.}}}  \\
\noindent
{\bf Abstract.} Building on results developed in \cite{LT24a}, where It\^{o}-SDEs with possibly degenerate and discontinuous dispersion coefficient and measurable drift were analyzed with respect to a given (sub-)invariant measure, we develop here additional elliptic regularity results for PDEs and consider the same equations with some further regularity assumptions on the coefficients to provide a pointwise analysis for every starting point in Euclidean space, $d\ge 2$. Our main result is (weak) well-posedness, i.e. weak existence and uniqueness in  law, which we obtain under our main  assumption for any locally bounded drift and arbitrary starting point among all solutions that spend zero time at the points of degeneracy of the dispersion coefficient. The points of degeneracy form a $d$-dimensional Lebesgue measure zero set, but may be hit by the weak solutions. Weak existence for arbitrary starting point is obtained under broader assumptions. In particular, in that case the drift does not need to be locally bounded.


\section{Introduction} \label{intro}
This paper is devoted to the study of existence and uniqueness in law of weak solutions, namely well-posedness, for possibly degenerate time-homogeneous It\^{o}-SDEs with measurable coefficients in dimension $d\ge 2$. First note, that in contrast to the case of continuous coefficients, there are already elementary examples of SDEs with discontinuous coefficients where well-posedness fails to hold. Therefore, obtaining well-posedness results for general classes of SDEs with both degenerate and discontinuous coefficients appears to be impossible without further restriction. Indeed, even in the globally uniformly strictly elliptic case with zero drift there is an example of an SDE with discontinuous diffusion coefficient where uniqueness in law does not hold (see \cite{Nadi}, \cite{S99}, or also \cite[Example 1.24]{CE05}). In order to obtain well-posedness, certain types of discontinuities were studied in the literature. For instance, the case where the diffusion matrix is not far from being continuous, i.e. continuous up to a small set (e.g. a countable set or a set of $\alpha$-Hausdorff measure zero with sufficiently small $\alpha$, or else, see for instance \cite{CeEsFa91}, or introductions of \cite{Kry04} and \cite{Nadi}  for references). In \cite{BaPa} the authors assumed the diffusion matrix to be piecewise constant on a decomposition of the Euclidean space into a finite union of polyhedrons and the drift was assumed to consist of locally bounded  functions with at most linear growth at infinity. Furthermore, in \cite{Gao} the author considered discontinuities of the diffusion matrix along the common boundary of the upper and lower half spaces. In the case where the diffusion matrix is degenerate  well-posedness for a special class of coefficients was obtained in  \cite{Kry04}. Though the discontinuity is only along a hyperplane of codimension one and the coefficients are quite regular outside the hyperplane, it seems to be one of the first examples of a discontinuous degenerate diffusion coefficient where well-posedness still holds (\cite[Example 1.1 and Theorem 2.18]{Kry04}).  \\
In this work, we follow a different approach. Instead of investigating coefficients which are not far from being continuous, we try to keep the conditions on the coefficients as weak as possible, but as in  \cite{LT24a}, we put constraints on the solutions. In \cite{LT24a}, we considered merely $L^2_{loc}$-assumptions on the drift, but admissible solutions were only searched for in a certain class of strong Markov processes with respect to a given sub-invariant measure. Here, generally speaking, any weak solution is allowed as long as it does not spend a positive amount of time at the zeros of the dispersion coefficient, we hence derive conditional uniqueness. Such kind of uniqueness appears often naturally, or as technical condition,  and has been investigated in many different cases, with different constraints in the literature (see for instance \cite{ES91,PavShe20,HTV03,RoeZh23} and a more detailed discussion of these papers below). \\
Now in order to describe our main results, which are summarized in concise, quickly readable form in Section \ref{themainresult} below (we suggest to read these first), let us first describe the nature of the dispersion coefficient $\widehat{\sigma}=\sqrt{\frac{1}{\psi}}\cdot \sigma$ that appears in our main Theorem \ref{maintheorem} about well-posedness.  Theorem \ref{maintheorem} is derived under any non-explosion condition and the condition {\bf (C)} of Section \ref{themainresult}.
As $A=\sigma\sigma^T$ appearing there, is supposed to be (pointwisely) locally uniformly strictly elliptic, the degeneracy of $\widehat{\sigma}$ can only stem from the factor $\sqrt{\frac{1}{\psi}}$,  while the discontinuity of $\widehat{\sigma}$ may stem from both the discontinuity of $\sqrt{\frac{1}{\psi}}$ and $A=(a_{ij})_{1\le i,j\le d}$ (see Remark \ref{maindifferences}(a)). Typical examples for $\sqrt{\frac{1}{\psi}}$ are given as for instance in condition (d) of Example \ref{scopepsi} or even (in the spirit of \cite{BaPa}) by $\sqrt{\frac{1}{\psi}}=\sum_{i=1}^\infty \alpha_i 1_{A_i}$, with $\R^d=\dot{\cup}_{i=1}^{\infty}A_i$, $(\alpha_i)_{i\in \N}\subset (0,\infty)$, for measurable $A_i$, such that $\psi$ satisfies the conditions of  {\bf (C)}. Moreover, condition \eqref{zerospen}, i.e. that an admissible solution must spend zero time at the zeros of $\sqrt{\frac{1}{\psi}}$ is necessary as can be seen from Remark \ref{remex}. Therefore our uniqueness in  law result is conditional among all solutions that spend zero time at the zeros of $\sqrt{\frac{1}{\psi}}$, but of course the condition is automatically fulfilled if $\sqrt{\frac{1}{\psi}}$ has no zeros, i.e. if it only attains strictly positive values.\\
Our method to obtain weak existence strongly differs from the techniques used in  \cite{BaPa} and \cite{Kry04} and in previous literature. Our techniques involve semigroup theory, elliptic and parabolic regularity theory, and the theory of generalized Dirichlet forms (i.e. the construction of a Hunt process from a sub-Markovian $C_0$-semigroup of contractions on some $L^1$-space with a weight). In the frame of generalized Dirichlet forms this program has been first carried out in \cite{LT18,LT19} and then among others refined and summarized in \cite{LST22}. Here, the existence of an $L^1(\R^d,\widehat{\mu})$-semigroup and resolvent, and generalized Dirichlet form follows directly from previous work \cite{LT24a}, since {\bf (C1)} $\Rightarrow$\cite[{\bf (A)}]{LT24a} for a certain density $\rho$ (see Remark \ref{aboutrho}). In order to show that the $L^1(\R^d,\widehat{\mu})$-semigroup induces the transition function of a Hunt process that weakly solves the underlying SDE, we were able to improve the elliptic and parabolic regularity results from \cite{LT18,LT19,LST22} by using recent regularity results of \cite{KK19}, which together with results of \cite{DK11} can be used to derive our main elliptic regularity result (see Theorem \ref{mainregulthm}  below). The weak solution that we construct satisfies condition \eqref{spendingzero} (cf. Remark \ref{casemdandindepbor}(i)). Moreover, our main elliptic regularity result is used to relax significantly the conditions on the coefficients imposed in \cite{LT19de} (cf. Remark \ref{maindifferences}).\\
In order to derive uniqueness in law, we adapt an idea of Stroock and Varadhan (\cite{StrVar}) to show uniqueness for the martingale problem using a Krylov type estimate. 
Krylov type estimates have been widely used to obtain the weak solution and its uniqueness, in particular pathwise uniqueness, simultaneously. In this work, weak existence of a solution and uniqueness in law are shown separately of each other using different techniques. We use a local Krylov type estimate (Theorem \ref{krylovtype}) to show uniqueness in law and once uniqueness in law holds we can complement the local Krylov type estimate at least for the time-homogeneous case (cf. last part of Theorem \ref{maintheorem}). In particular our method typically implies weak existence results that are more general than the uniqueness results (compare for instance the weak existence result stated in Theorem \ref{maintheorem} and the one of Corollary \ref{weakexcoro}). As further byproduct of the improved regularity results obtained in this paper, Theorem \ref{weakexistence4}  generalizes substantially the weak existence result \cite[Theorem 3.22]{LST22}.\\
Let us now mention further related work. Engelbert and Schmidt studied in \cite{ES91} general one-dimensional It\^{o}-SDEs including the case where the dispersion coefficient is degenerate and a similar condition to \eqref{zerospen}, called fundamental solution (see \cite[Definition 4.16]{ES91})), is considered to obtain well-posedness.
\cite{PavShe20} investigates the well-posedness of the Stratonovich counterpart of the Girsanov SDE and the condition of spending zero time at $0$ is also imposed on the admissible solutions. In  \cite{HTV03} equations without drift and with bounded dispersion coefficient that is locally of bounded variation are studied, and the Brownian motion is replaced by a regular H\"{o}lder path. In this framework an analogue of the one-dimensional Girsanov SDE (i.e. \eqref{generalizedGirsanov} with $\widehat{\mathbf{H}}\equiv0$ and $d=1$) is described in \cite[Example 3.3]{HTV03}. We further mention the results \cite{Fi08, Tre} which show  the well-posedness of the martingale problem by developing the well-posedness of Fokker-Planck equations and the superposition principle.
In particular, \cite[Theorems 3.1, 3.3]{Tre} covers a large class of diffusion and drift coefficients including time-dependent cases, but the twice weak spatial differentiability of the diffusion coefficient is imposed.
Therefore, \cite[Theorems 3.1, 3.3]{Tre}  can not yield the well-posedness of \eqref{underlyingsde} if the entries of $\widehat{\sigma}$ are not twice weakly differentiable. Furthermore, in \cite{Tre} global integrability conditions are imposed. Another result related to ours, at least in what concerns the $VMO$-assumption on the diffusion coefficient, can be found in \cite[Theorem 7.2]{Kr21} where a non-degenerate $VMO$ diffusion coefficient and an $L^d$-drift coefficient is considered under global assumptions. 
Concerning a recent result on singular SDEs with throughout global assumptions around the Ladyzhenkaya-Prodi-Perrin condition, we mention  \cite[Theorem 1.1]{RoeZh23}, where a conditional well-posedness result is obtained, i.e. weak existence and uniqueness in law are shown among all the solutions that satisfy a Krylov type estimate. \\
This paper is organized as follows: in Section \ref{framework} we introduce notations and conventions which are in particular necessary to formulate and understand the main results that are summarized in short form in Section \ref{themainresult}. Section \ref{fewfore} is devoted to the semigroup and resolvent regularity. Section \ref{effwegewe} is the core of our paper. It contains the results about weak existence and uniqueness in law.  In Section \ref{someanlyticresults} we develop the new regularity results that are in particular based on \cite{DK11,KK19}. Throughout the paper, we systematically develop results that are based on the conditions {\bf (C1)-(C3)}, and {\bf (C)} from the weakest to the strongest assumption. 
Thus the following implications related to the conditions that are occurring in this paper (and \cite{LT24a}) hold:
\[
 {\bf (C)} \Rightarrow {\bf (C3)} \Rightarrow {\bf (C2)} \Rightarrow {\bf (C1)} (\Rightarrow \cite[{\bf (A)}]{LT24a} \text{ for a certain } \rho\ \text{(see Remark \ref{aboutrho}})).
 \]
\section{Notations and conventions}\label{framework}
Notations which are not mentioned in this text, will be those as defined in \cite[Notations and Conventions]{LST22}.
The notations in the following definition will be often used in the sequel.
\begin{defn} \label{basidefn}  
\begin{itemize}
\item[(i)]For $\varphi \in L^1_{loc}(\mathbb{R}^d)$,with $\varphi>0$ a.e., $\frac{1}{\varphi}$ denotes any Borel measurable function satisfying $\varphi \cdot \frac{1}{\varphi}=1$ a.e.
\item[(ii)] For a (possibly non-symmetric) matrix of locally integrable functions $B = (b_{ij})_{1 \leq i,j \leq d}$ on $\mathbb{R}^d$ and a vector field $\bold{E}=(e_1,\ldots, e_d) \in L^1_{loc}(\mathbb{R}^d, \mathbb{R}^d)$, we write ${\rm{div}}B = \bold{E}$ if
\begin{equation} \label{defndivma}
\int_{\mathbb{R}^d} \sum_{i,j=1}^d b_{ij} \partial_i \phi_j \, dx = - \int_{\mathbb{R}^d} \sum_{j=1}^d e_j \phi_j \, dx \quad \text{ for all $\phi_j \in C_0^{\infty}(\mathbb{R}^d)$, $j=1, \ldots, d$}.
\end{equation}
In other words, if we consider the family of column vectors
\[
\bold{b}_j=\begin{pmatrix} 
b_{1j} \\ \vdots \\  b_{dj}
\end{pmatrix},\quad 1\le j\le d,\  \text{ i.e.}\  B=(\bold{b}_1 | \ldots |\bold{b}_d),
\]
then \eqref{defndivma} is equivalent to 
\begin{equation*}
\int_{\mathbb{R}^d} \sum_{j=1}^d \langle \bold{b}_j, \nabla \phi_j \rangle dx =-\int_{\mathbb{R}^d} \langle \bold{E}, \Phi \rangle dx, \text{ \;\;\; for all $\Phi=(\phi_1, \ldots, \phi_d) \in C_0^{\infty}(\mathbb{R}^d)^d$.}
\end{equation*}
Let $B^T=(b^T_{ij})_{1 \leq i,j \leq d}=(b_{ji})_{1 \leq i,j \leq d}$ denote the transpose of $B$. For $\rho \in H^{1,1}_{loc}(\mathbb{R}^d)$ with $\rho>0$ a.e., we write
\begin{equation*}\label{logder 2}
\beta^{\rho, B^T} := \frac12{\rm{div}} B + \frac{1}{2\rho} B^T \nabla \rho.
\end{equation*}
For $\psi \in L^1_{loc}(\mathbb{R}^d)$, $\psi>0$ a.e. with $\frac{1}{\psi} \in L^{\infty}_{loc}(\mathbb{R}^d)$, we write
\begin{equation*}\label{logder 1}
\beta^{\rho, B^T, \psi}:= \frac{1}{\psi} \beta^{\rho, B^T} = \frac{1}{2\psi} {\rm{div}} B + \frac{1}{2\psi\rho }  B^T \nabla \rho
\end{equation*}
\item[(ii)]
For $B=(b_{ij})_{1 \leq i,j \leq d}$, with $b_{ij} \in H^{1,1}_{loc}(\mathbb{R}^d)$, $1\le i,j\le d$, we define on $\mathbb{R}^d$ a vector field $\nabla B =\big(  (\nabla B)_1,  \ldots, (\nabla B)_d \big) $ by
$$
(\nabla B)_i = \sum_{j=1}^d \partial_j b_{ij}, \quad \; i=1, \ldots, d.
$$
In particular, by \cite[Remark 2.2(i)]{LT24a}, ${\rm div} B=\nabla B^T$ if $b_{ij} \in H^{1,1}_{loc}(\mathbb{R}^d)$. $1\le i,j\le d$.
\end{itemize}
\end{defn}

\begin{defn}
$VMO$ denotes the set of all functions of vanishing mean oscillation, i.e. the set of all $g \in L^1_{loc}(\mathbb{R}^d)$ such that 
$$
\sup_{z \in \mathbb{R}^d, r<R} r^{-2d}  \int_{B_r(z)} \int_{B_r(z)} |g(x)-g(y)| \,dx dy \leq  \omega(R),
$$
for all $R>0$, where $\omega$ is a positive continuous function on $[0, \infty)$ depending on $g$ and $\omega(0)=0$. For each open ball $B$, denote by $VMO(B)$ the set of all functions $f \in L^1(B)$ such that $\widetilde{f}|_B= f$ for some $\widetilde{f} \in VMO$. Denote by $VMO_{loc}$ the set of all functions $h \in L^1_{loc}(\mathbb{R}^d)$ such that $h|_B \in VMO(B)$ for each open ball $B$ in $\mathbb{R}^d$.
\end{defn}
By the Poincar\'{e} inequality $W^{1,d}(\mathbb{R}^d) \subset VMO$ and by the uniform continuity of functions in $C_0(\mathbb{R}^d)$, we get $C_0(\mathbb{R}^d) \subset VMO$. Moreover, by extension arguments, it holds that $W^{1,d}_{loc}(\mathbb{R}^d) \cup C(\mathbb{R}^d) \subset VMO_{loc}$.\\
\centerline{}
\noindent
Additionally, let $q \in [1, \infty]$, $U$ be an open set in $\R^d$ and $\nu$ be a measure $\nu$ on $\mathcal{B}(\mathbb{R}^d)$. Define $L^q(U,\R^d, \nu):=\{ \bold{F}=(f_1,\dots,f_d) : \, f_i\in L^q(U, \nu), 1\le i\le d\}$, equipped with the norm, $\| \mathbf{F} \|_{L^q(U, \R^d,\nu)}  := \| \|\mathbf{F}\| \|_{L^q(U, \nu)}, \,  \mathbf{F} \in L^q(U, \R^d, \nu)$. 
Denote $L^q(U, \mathbb{R}^d, dx)$ by $L^q(U, \mathbb{R}^d)$ with the norm
$$
\|\bold{F}\|_{L^q(U)}  =\| \bold{F} \|_{L^q(U, \mathbb{R}^d, dx)}.
$$
Let $I$ be an open interval on $\R$ and $p, q \in [1, \infty]$. Denote by $L^{p,q}(U \times I)$ the space of all Borel measurable functions $f$ on $U \times I$ for which 
$$
\|f\|_{L^{p,q}(U \times I)}:=\| \|f(x,  t )\|_{L^{p}(U, dx)} \|_{L^{q}(I, dt)}< \infty,
$$
For a locally integrable function $g$ on $U \times I$ and  $i \in \{1, \dots d  \}$, we denote by $\partial_i g$ the $i$-th weak spatial derivative on $U \times I$ and by $\partial_t g$ the weak time derivative on  $U \times I$, provided it exists. 
For $r \in [1, \infty)$, denote by $H^{2,r}(U)$ the set of all functions $u \in L^r(U)$ such that $\partial_i u, \partial_{i} \partial_j u \in L^r(U)$ for all $1 \leq i,j \leq d$ equipped with the norm
$
\|u\|_{H^{2,r}(U)}:= \big( \|u\|^r_{L^r(U)} +\sum_{i=1}^d \|\partial_i u \|^r_{L^r(U)} + \sum_{i,j=1}^d  \|\partial_i \partial_j u\|^r_{L^r(U)} 	\big )^{1/r}.
$
For $p,q  \in [1, \infty]$, let $W^{2,1}_{p,q}( U \times I)$ be the set of all locally integrable functions
$g: U \times I \rightarrow \R$ such that $\partial_t g,\, \partial_i g,\, \partial_i \partial_j g \in L^{p,q}( U \times I)$ for all $1 \leq i,j \leq d$. We further let  $L^{p}(U \times I):=L^{p,p}(U \times I)$, and $W^{2,1}_{p}( U \times I ):= W^{2,1}_{p,p}( U \times I)$.

\section{The main results}\label{themainresult}
\subsection{Well-posedness}
Let us consider the following set of conditions: \\ \\
{\bf (C)}: $d\ge 2$, $\widehat{\mathbf{H}}:\R^d\to \R^d$ is measurable and (pointwisely) locally bounded. $A =(a_{ij})_{1 \leq i,j \leq d}$ is a symmetric matrix of measurable functions which is (pointwise everywhere) locally uniformly strictly elliptic, i.e. for each open ball $B$, there exist strictly positive constants $\lambda_B$ and $\Lambda_B>0$ such that 
\begin{equation} \label{uniellip}
\lambda_{B} \|\xi\|^2 \leq \langle A(x) \xi, \xi \rangle \leq \Lambda_B \|\xi\|^2, \quad \text{ for all } \xi \in \R^d\, \text{ and } \;  \text{all } x \in B.
\end{equation}
$a_{ij} \in VMO_{loc}$ for each $1 \leq i,j \leq d$ with $\text{div}A \in L^{d+1}_{loc}(\mathbb{R}^d, \mathbb{R}^d)$. Moreover, $\psi \in L^q_{loc}(\mathbb{R}^d)$, where $q \in (d+1, \infty]$ is fixed, $\psi >0$ a.e. and $\frac{1}{\psi}$ is (pointwisely) locally bounded with $\frac{1}{\psi}(x)\ge 0$ for any $x\in \R^d$.\\ \\
The following is our main theorem on well-posedness and it is a direct consequence of Propositions \ref{resolkrylov} and \ref{conservativenesscond}, Theorem \ref{weakexistence4}(ii) (the main result about weak existence, which is more general than what is stated in the theorem) and Theorem \ref {weifjowej} (the main result about uniqueness in law) below, choosing  $\mathbf{H}=\widehat{\mathbf{H}}-\frac{1}{2\psi} \text{div}A \in L^{d+1}_{loc}(\R^d, \mathbb{R}^d)$ in {\bf (C3)} of Section \ref{Section 6}.
\begin{theorem}\label{maintheorem}
Assume {\bf (C)}. Moreover, assume that there exist constants $N_0 \in \N$ and $M>0$ such that
\begin{eqnarray}\label{conscondit}
-\frac{\langle \frac{1}{\psi}(x)A(x)x, x \rangle}{\left \| x \right \|^2}+ \frac{\mathrm{trace}(\frac{1}{\psi}(x)A(x))}{2}+ \big \langle\widehat{\mathbf{H}} (x), x \big \rangle  \leq M \|x\|^2 (\ln\|x\|  +1), \quad \text{a.e. on}\  \mathbb{R}^d \setminus \overline{B}_{N_0},
\end{eqnarray}
or more generally that the Hunt process $\M$ occurring in Theorem \ref{weakexistence4} is non-explosive. Let $(\sigma_{ij})_{1 \le i,j \le d}$ be a matrix of measurable functions such that $\sigma(x) \sigma(x)^T =A(x)$ for all $x \in \mathbb{R}^d$.
Set 
$$
\widehat{\sigma}=\sqrt{\frac{1}{\psi}}\cdot \sigma \;\; \text{ , i.e. }\widehat{\sigma}_{ij}=\sqrt{\frac{1}{\psi}}\cdot \sigma_{ij}, \;\;\; 1 \leq i,j \leq d.
$$
Let $y \in \mathbb{R}^d$ be arbitrary. Then, there exists a filtered probability space  $(\Omega, (\F_t)_{t \ge 0}, \F, \P_{y})$  carrying a $d$-dimensional Brownian motion $(W_t)_{t \ge 0}$ and a (strong Markov) weak solution $(X_t)_{t \ge 0}$ to 
\begin{equation} \label{underlyingsde}
X_t = y+  \int_0^t \widehat{\sigma} (X_s) \, dW_s+   \int^{t}_{0}  \widehat{\mathbf{H}} (X_s) \, ds, \;\; \quad 0\le  t <\infty
\end{equation}
such that
\begin{equation} \label{spendingzero}
\int_0^{\infty} 1_{\{\sqrt{\frac{1}{\psi}}=0\}}(X_s) ds =0, \quad \text{ $\P_{y}$-a.s}. 
\end{equation}
Uniqueness in law holds in the following sense: if there exists a filtered probability space
$(\widetilde{\Omega}, (\widetilde{\F}_t)_{t \ge 0}, \widetilde{\F}, \widetilde{\P}_{y})$ carrying a $d$-dimensional Brownian motion $(\widetilde{W}_t)_{t \ge 0}$ and a weak solution $(\widetilde{X}_t)_{t \ge 0}$ to \eqref{underlyingsde} such that
\begin{equation} \label{zerospen}
\int_0^{\infty} 1_{\{\sqrt{\frac{1}{\psi}}=0\}}(\widetilde{X}_s) ds =0, \quad \text{ $\widetilde{\P}_{y}$-a.s.},
\end{equation}
then
$$
\P_y \circ X^{-1} = \widetilde{\P}_y \circ \widetilde{X}^{-1}\ \text{ on } \ \ \mathcal{B}(C([0, \infty), \R^d)).
$$
Moreover, for any weak solution $(\widetilde{X}_t)_{t \ge 0}$ to \eqref{underlyingsde} as above starting form $y$ such that \eqref{zerospen} holds, the Krylov type estimate of Theorem \ref{krylovtype} can be complemented as follows: for any $g \in L^r(\mathbb{R}^d, \widehat{\mu})$, where $r \in [s, \infty]$ with $s \in (\frac{d}{2}, \infty)$ fixed such that $\frac{1}{q} + \frac{1}{s}<\frac{2}{d}$, $\frac{1}{\infty}:=0$, and where $\widehat{\mu}$ is as in Remark \ref{aboutrho} (see in particular the discussion in (ii) there), there exists a constant $C_{t,y,r}>0$ depending on $t,y, r$, such that
$$
\widetilde{\E}_y \Big [ \int_0^t g(\widetilde{X}_s) \, ds \Big ] \le C_{t,y,r} \|g\|_{L^r(\R^d, \widehat{\mu})}.
$$ 

\end{theorem}

\begin{remark} 
In Theorem \ref{maintheorem}, the $d \times d$ matrix of functions $\sigma$ satisfying $\sigma \sigma^T = A$ can be replaced by a $d \times m$ matrix of functions with $m \geq d$ (see Remark \ref{casemdandindepbor}(ii)). \\
\end{remark}

\begin{remark}\label{maindifferences}
The main differences of Theorem \ref{maintheorem} and \cite[Theorem 13]{LT19de} are the following:
\begin{itemize}
\item[(a)]
We do not need to impose the condition that $a_{ij} \in H^{1,p}_{loc}(\mathbb{R}^d)$, $p>d$ nor $a_{ij} \in C(\R^d)$, $1 \leq i,j \leq d$. Instead $a_{ij}\in VMO_{loc}$ is assumed. For instance, let $\zeta \in L^{\infty}(\mathbb{R})\cap VMO_{loc}$ with $\zeta \geq c$ for some constant $c>0$ and assume that $\zeta$ is discontinuous and has no Sobolev differentiability (i.e. $\zeta \notin H^{1,1}_{loc}(\mathbb{R})$).
Moreover, let $a_{11}(x_1, \ldots, x_d)=\zeta(x_2)$, $a_{22}(x_1, \ldots, x_d)=\zeta(x_3), \ldots, a_{d-1 \, d-1}(x_1, \ldots, x_d)=\zeta(x_d)$, $a_{dd}(x_1, \ldots, x_d)=\zeta(x_1)$ and $a_{ij}=0$ if $i \neq j$. Then $a_{ii} \notin H^{1,1}_{loc}(\mathbb{R}^d)$ but ${\rm{div}}  A = 0$.  In particular, in contrast to  \cite{LT19de}, the $a_{ij}$ do not necessarily need to have continuous versions, nor to be in $H^{1,1}_{loc}(\mathbb{R}^d)$.

\item[(b)]
The previous condition that $a_{ij}\in H^{1,2d+2}_{loc}(\mathbb{R}^d)$, $1\le i,j \le d$, and that $\psi \in L^q_{loc}(\mathbb{R}^d)$, $q>2d+2$ is relaxed to the one that ${\rm{div} } A\in L^{d+1}_{loc}(\mathbb{R}^d, \R^d)$, and $\psi \in L^q_{loc}(\mathbb{R}^d)$, $q>d+1$.

\item[(c)] Uniqueness in law can be considered for each fixed initial condition $y\in \R^d$ separately, in contrast to \cite{LT19de} where uniqueness in law must be considered for all initial conditions $y\in \R^d$ simultaneously.
\end{itemize}
\end{remark}

\subsection{Weak existence}
As already indicated in the paragraph right before Theorem \ref{maintheorem},  weak existence can be realized under more general assumptions than {\bf (C)}. The most general are presented in Theorem \ref{weakexistence4}(ii), but are somehow difficult to catch at first sight. Here we present a special case of Theorem \ref{weakexistence4}(ii) (see also Example \ref{forexiste}(iv)) in order to provide for the convenience of the reader an easily accessible weak existence result.

\begin{cor}[of Theorem \ref{weakexistence4}]\label{weakexcoro}
Let $A=(a_{ij})_{1 \leq i,j \leq d}$ be a symmetric matrix of measurable functions which is (pointwise almost everywhere) locally uniformly strictly elliptic (for the corresponding definition see condition {\bf (C1)} of Section \ref{fewfore} below).
Let further $d,p,q,s$ and $\psi$ be as in {\bf (C1)}, $\widehat{\bold{H}} \in L_{loc}^{\frac{sq}{q-1}}(\mathbb{R}^d, \mathbb{R}^d)$, $\psi \widehat{\bold{H}} \in L^p_{loc}(\mathbb{R}^d, \mathbb{R}^d)$, and ${\rm div} A \in L^p_{loc}(\mathbb{R}^d, \mathbb{R}^d)$.  Assume that \eqref{conscondit} holds or more generally that the Hunt process $\M$ occurring in Theorem \ref{weakexistence4} is non-explosive. Let $(\sigma_{ij})_{1 \le i,j \le d}$ be any matrix of measurable functions such that $\sigma(x) \sigma(x)^T =A(x)$ for a.e. $x \in \mathbb{R}^d$. Let $y \in \mathbb{R}^d$. Then, there exists a filtered probability space  $(\Omega, (\F_t)_{t \ge 0}, \F, \P_{y})$  carrying a $d$-dimensional Brownian motion $(W_t)_{t \ge 0}$ and a (strong Markov) weak solution $(X_t)_{t \ge 0}$ to \eqref{underlyingsde} satisfying \eqref{spendingzero}. Moreover, for any $g \in L^r(\mathbb{R}^d, \widehat{\mu})$ with $r \in [s, \infty]$, the Krylov type estimate given at the end of Theorem \ref{maintheorem} holds for this solution. 
\end{cor}

\section{Existence of an infinitesimally invariant measure and strong Feller properties} \label{fewfore}

Here we state the following set of conditions which will be assumed in this section.
\begin{itemize}
\item[\bf (C1)]
For some $d \geq 2$, $q \in (\frac{d}{2}, \infty]$ is fixed with $q \geq 2$. $s \in (\frac{d}{2}, \infty)$ is fixed with $\frac{1}{q} + \frac{1}{s}<\frac{2}{d}$, $\frac{1}{\infty}:=0$.
$\psi \in L^q_{loc}(\R^d)$ satisfies $\psi>0$ a.e. on $\mathbb{R}^d$ and $\frac{1}{\psi} \in L_{loc}^{\infty}(\R^d)$. $p \in (d, \infty)$ is fixed and $\psi \bold{H} \in L^p_{loc}(\mathbb{R}^d, \mathbb{R}^d)$ where $\bold{H}$ is a measurable vector field on $\mathbb{R}^d$.
$A=(a_{ij})_{1 \leq i,j \leq d}$ is a symmetric matrix of measurable functions which is (pointwise almost everywhere) locally uniformly strictly elliptic, i.e. for each open ball $B$, there exist strictly positive constants $\lambda_B$ and $\Lambda_B>0$ such that 
\begin{equation} \label{uniellipa.e.}
\lambda_{B} \|\xi\|^2 \leq \langle A(x) \xi, \xi \rangle \leq \Lambda_B \|\xi\|^2, \quad \text{ for all } \xi \in \R^d\, \text{ and a.e. } x \in B.
\end{equation}
and $\text{div} A \in L_{loc}^2(\mathbb{R}^d, \mathbb{R}^d)$. $C=(c_{ij})_{1 \leq i,j \leq d}$, with $c_{ij}\in L_{loc}^{\infty}(\R^d)$, $1\le i,j\le d$, is an anti-symmetric matrix (i.e. $c_{ij} = -c_{ji}$) such that $\text{div}\, C \in L^2_{loc}(\mathbb{R}^d, \mathbb{R}^d)$. For each $1 \leq i,j \leq d$, $a_{ij}$,  $c_{ij} \in VMO_{loc}$, and the partial differential operator $(L, C_0^{\infty}(\mathbb{R}^d))$ is defined by
\begin{align*}
Lf  & = \frac{1}{2 \psi } \text{div} \left(  (A+C) \nabla f  \right)  + \langle \mathbf{H}, \nabla f \rangle, \\
&=\frac12 \text{trace} \big(\widehat{A} \nabla^2 f \big) + \left \langle \mathbf{G}, \nabla f \right \rangle,  \; \quad f \in C_0^{\infty}(\mathbb{R}^d),
\end{align*}
where
$$
\widehat{A} = \frac{1}{\psi} A \; \text{ and }\; \mathbf{G} = \frac{1}{2 \psi} \text{div}(A +C)    + \mathbf{H}.
$$ 
\end{itemize}

\begin{theorem} \label{helholmop}
Under the assumption {\bf (C1)}, there exists $\rho \in H_{loc}^{1,p}(\R^d) \cap C(\R^d)$ satisfying $\rho(x)>0$ for all $x \in \R^d$ such that $\mathbf{G}-\beta^{\rho, A, \psi} \in L^2_{loc}(\mathbb{R}^d, \mathbb{R}^d, \rho \psi dx)$, $\rho \psi  \left( \mathbf{G}-\beta^{\rho, A, \psi}  \right) \in L_{loc}^2(\R^d, \R^d)$ (equivalently $\psi\left (\mathbf{G}-\beta^{\rho, A, \psi}\right ) \in L^2_{loc}(\mathbb{R}^d, \mathbb{R}^d, \rho dx)$) and furthermore
\begin{equation} \label{helm}
\int_{\R^d} \langle \mathbf{G}-\beta^{\rho, A, \psi}, \nabla \varphi \rangle \rho \psi dx =0, \quad \text{for all } \varphi \in C_0^{\infty}(\R^d).
\end{equation}
In particular, under the assumption {\bf (C1)},  assumption \cite[{\bf (A)}]{LT24a} holds for $\widehat{\mu} = \rho \psi dx$ with $\rho$ as in this theorem, $\mathbf{B}= \mathbf{G}-\beta^{\rho, A, \psi}$ and
$$
Lf = L^0 f + \langle \mathbf{B}, \nabla  f \rangle,  \; \quad f \in C_0^{\infty}(\mathbb{R}^d),
$$
where
$$
 L^0 f=\frac12{\rm trace}(\widehat{A} \nabla^2 f) + \langle \beta^{\rho, A, \psi}, \nabla f \rangle.
$$
Thus, by \cite[Theorem 3.5]{LT24a}, there exists a closed extension $(\overline{L}, D(\overline{L}))$ of $(L, C_0^{\infty}(\mathbb{R}^d))$ which generates a sub-Markovian $C_0$-semigroup of contractions $(T_t)_{t>0}$ and a sub-Markovian $C_0$-resolvent of contractions $(G_{\alpha})_{\alpha>0}$ on $L^1(\mathbb{R}^d, \widehat{\mu})$. Both, $(T_t)_{t>0}$ and  $(G_{\alpha})_{\alpha>0}$ when restricted to $L^1(\mathbb{R}^d, \widehat{\mu})_b$ can be extended 
to a sub-Markovian $C_0$-semigroup of contractions and a sub-Markovian $C_0$-resolvent of contractions on each $L^r(\R^d, \widehat{\mu})$, $r\in [1,\infty)$, respectively. In particular, $(T_t)_{t >0}|_{L^1(\mathbb{R}^d, \widehat{\mu})_b}$ and  $(G_{\alpha})_{\alpha>0}|_{L^1(\mathbb{R}^d, \widehat{\mu})_b}$ extend to a sub-Markovian semigroup and a sub-Markovian resolvent of contractions on $L^{\infty}(\mathbb{R}^d, \widehat{\mu})$, denoted again by $(T_t)_{t>0}$ and $(G_{\alpha})_{\alpha>0}$, respectively.

\end{theorem}
\begin{proof}
By \cite[Theorem 2.27(i)]{LST22}, there exists $\rho \in H_{loc}^{1,2}(\mathbb{R}^d) \cap C(\mathbb{R}^d)$ satisfying $\rho(x)>0$ for all $x \in \R^d$ such that 
\begin{equation} \label{basicintei}
\int_{\R^d}\big \langle \frac12 (A+C^T) \nabla \rho - \rho \psi \bold{H}, \nabla\varphi \big \rangle dx =0, \quad \text{ for all } \varphi \in C_0^{\infty}(\R^d).
\end{equation}
Since it holds that
$$
\int_{\R^d}\big \langle \frac12 (A+C^T) \nabla \rho, \nabla\varphi \big \rangle dx =\int_{\mathbb{R}^d}\langle \rho \psi \bold{H}, \nabla \varphi \rangle dx, \quad \text{ for all } \varphi \in C_0^{\infty}(\R^d),
$$
we obtain $\rho \in H^{1,p}_{loc}(\mathbb{R}^d)$ by Proposition \ref{basicpropvm}(iii), Proposition \ref{vmoprop}(ii) and Theorem \ref{mainregulthm}(iii). Then, \eqref{basicintei} implies that  for any $\varphi \in C_0^{\infty}(\mathbb{R}^d)$
\begin{align*}
0&= \int_{\R^d}\big \langle \frac12 A \nabla \rho - \rho \psi \bold{H}, \nabla\varphi \big \rangle dx  + \int_{\mathbb{R}^d} \frac12  \langle \nabla \rho, C \nabla \varphi  \rangle dx \\
&= \int_{\R^d}\big \langle \frac12 A \nabla \rho - \rho \psi \bold{H}, \nabla\varphi \big \rangle dx  - \int_{\mathbb{R}^d} \frac12  \text{trace} (\rho C \nabla^2 \varphi)  +  \big\langle \frac12 \rho \, \text{div}\,C, \nabla \varphi \big \rangle dx \\
&= \int_{\mathbb{R}^d} \big \langle -\frac{1}{2\psi} \text{div} C - \mathbf{H} + \frac{1}{2 \rho \psi}A \nabla \rho, \nabla \varphi \big \rangle \rho \psi \,dx \\
&= \int_{\mathbb{R}^d}  \langle - \mathbf{G}  + \beta^{\rho, A, \psi}, \nabla \varphi \rangle \rho \psi \,dx.
\end{align*}
Thus for $\bold{B} := \bold{G} - \beta^{\rho, A, \psi} = \frac{1}{2\psi} \text{div}\,C + \mathbf{H} -\frac{1}{2 \rho \psi}A \nabla \rho$,  \eqref{helm} holds with $\rho \psi \mathbf{B} \in L_{loc}^2(\mathbb{R}^d, \mathbb{R}^d)$ and $\mathbf{B} \in L^2_{loc}(\mathbb{R}^d, \mathbb{R}^d, \rho \psi dx)$. Using in particular the results of \cite{LT24a}, the remaining statements follow immediately.
\end{proof}
\begin{remark}\label{aboutrho}
(i) If {\bf (C1)} holds for given $\psi$, $\bold{H}$, $A=(a_{ij})_{1 \leq i,j \leq d}$, $C=(c_{ij})_{1 \leq i,j \leq d}$, we pick and fix an arbitrary density $\rho$ as in Theorem \ref{helholmop} and hereby also fix $\mu=\rho\,dx$, $\widehat{\mu}=  \psi\,d\mu$. In particular, $\rho, \psi, A, \mathbf{B}$, where $\mathbf{B}$ is defined as in Theorem \ref{helholmop} fulfill \cite[{\bf (A)}]{LT24a} and therefore all results of \cite{LT24a} which hold under assumption  \cite[{\bf (A)}]{LT24a}, in particular those stated at the end of Theorem \ref{helholmop}, are applicable. \\
(ii) In contrast to \cite{LT24a} where $\rho$ is explicitly given in assumption {\bf (A)}, here $\rho$ is abstract and there may be several $\rho$ fulfilling the same assumptions as \cite[Theorem 2.27(i)]{LST22}, which is used in the proof Theorem \ref{helholmop}, is only about existence. Nonetheless, the reader may have noted that the first order perturbation $\mathbf{G}$ of the linear operator $L$ does not depend on $\rho$ and moreover all following statements are at least qualitatively locally equivalent for possibly different $\rho$, as each $\rho$ is continuous and locally bounded above and away from zero.
\end{remark}
\begin{theorem} \label{rescondojc}
Assume {\bf (C1)} and let $f \in \cup_{r \in [s,\infty]} L^r(\R^d,\widehat{\mu})$. Let $(G_{\alpha})_{\alpha>0}$ be as described in Theorem 3.1. Then $G_{\alpha} f$ has a locally H\"{o}lder continuous $dx$-version $R_{\alpha} f$ on $\R^d$. Furthermore for any open balls $U$, $V$ satisfying $\overline{U} \subset V$, we have the following estimate
\begin{equation} \label{resestko}
\|  R_{\alpha} f \|_{C^{0, \gamma}(\overline{U})} \le c_2 \left ( \| G_{\alpha}f \|_{L^1(V,\widehat{\mu})}+ \| f \|_{L^s(V, \widehat{\mu})}\right ),
\end{equation}
where $c_2>0$, $\gamma \in (0,1)$ are constants which only depend on $A$, $C$, $\psi$, $\bold{H}$, $U$ and $V$.
\end{theorem}
\begin{proof}
Let $f \in L^1(\mathbb{R}^d, \widehat{\mu})_b$ and $\alpha>0$. Since $G_{\alpha} f =\overline{G}_{\alpha}f \in D(\overline{L})_b \subset D(\mathcal{E}^0) \subset H^{1,2}_{loc}(\mathbb{R}^d)$ by \cite[Theorem 3.5(i)]{LT24a}, we get
\begin{eqnarray}
&&\int_{\R^d} \big \langle \frac12 \rho A \nabla G_{\alpha}f, \nabla \varphi \big \rangle dx - \int_{\R^d} \langle \rho \psi \mathbf{B}, \nabla G_{\alpha} f \rangle \varphi\, dx  +\int_{\R^d} (\alpha \rho \psi G_{\alpha} f) \,\varphi dx  \nonumber \\
&& = \int_{\R^d} (\rho \psi f) \,\varphi dx, \quad \; \text{ for all } \varphi \in C_0^{\infty}(\R^d). \label{feppwikmee}
\end{eqnarray}
We have
\begin{align*}
-\rho \psi \bold{B} = -\frac{\rho}{2} \text{div} (A+C) - \rho \psi \bold{H}+\rho \psi \beta^{\rho, A, \psi} = -\frac{\rho}{2} \text{div}\,C - \rho \psi \bold{H} + \frac{1}{2} A \nabla \rho.
\end{align*}
If $\varphi \in C_0^{\infty}(\mathbb{R}^d)$ is arbitrary but fixed, there exists an open ball $B$ with  with $\text{supp}\,\varphi \subset B$.
Choose $(\rho_m)_{m \geq 1}$ and $(v_n)_{n \geq 1} \subset C^{\infty}(\overline{B})$ such that  $\lim_{m \rightarrow \infty} \rho_m=\rho$ in $H^{1,p}(B)$ and $\lim_{n \rightarrow \infty} v_n =G_{\alpha} f$ in $H^{1,2}(B)$. Then, using the anti-symmetry of $C^T$ and integration by parts, we obtain 
\begin{align*}
&\int_{\mathbb{R}^d} \big \langle -\frac{\rho}{2} \text{div}\, C, \nabla G_{\alpha} f \big \rangle \varphi dx = \lim_{m \rightarrow \infty} \lim_{n \rightarrow \infty} \int_{\mathbb{R}^d} \big \langle -\frac{\rho_m}{2} \text{div}\,C, \nabla v_n \big \rangle \varphi dx \\
&= \lim_{m \rightarrow \infty} \lim_{n \rightarrow \infty} \int_{\mathbb{R}^d} \Big (\frac12 \text{trace}(\rho_m \varphi C^T \nabla^2 v_n) +\frac{1}{2} \big \langle C^T \nabla (\rho_m \varphi), \nabla v_n \big \rangle \Big ) \,dx \\
&= \int_{\mathbb{R}^d} \frac{1}{2} \big \langle C^T \nabla (\rho \varphi), \nabla G_{\alpha} f   \big \rangle\,dx = \int_{\mathbb{R}^d}  \Big (\big \langle \frac{1}{2} \rho C \nabla G_{\alpha} f, \nabla \varphi  \big \rangle dx +\big \langle  \frac12 C^T\nabla \rho, \nabla G_{\alpha} f   \big \rangle \varphi\Big ) dx.
\end{align*}
Hence, for each $\varphi \in C_0^{\infty}(\mathbb{R}^d)$ it holds that
\begin{align*}
&-\int_{\mathbb{R}^d} \langle \rho \psi \mathbf{B}, \nabla G_{\alpha} f  \rangle \varphi dx =  
\int_{\mathbb{R}^d} \big \langle \frac{1}{2} \rho C \nabla G_{\alpha} f, \nabla \varphi  \big \rangle dx + \int_{\mathbb{R}^d} \langle  - \rho \psi \bold{H} +  \frac12 (A+C^T)\nabla \rho, \nabla G_{\alpha} f   \rangle \varphi dx.
\end{align*}
Thus, \eqref{feppwikmee} implies that
\begin{eqnarray} 
&&\int_{\R^d} \big \langle \frac12 \rho (A+C) \nabla G_{\alpha}f, \nabla \varphi \big \rangle dx + \int_{\R^d} \langle \widehat{\bold{F}}, \nabla G_{\alpha} f \rangle \varphi\, dx  +\int_{\R^d} (\alpha \rho \psi G_{\alpha} f) \,\varphi dx  \nonumber \\
&& = \int_{\R^d} (\rho \psi f) \,\varphi dx, \quad \; \text{ for all } \varphi \in C_0^{\infty}(\R^d), \label{feppwikmee2}
\end{eqnarray}
where $\widehat{\bold{F}} = - \rho \psi \bold{H} +  \frac12 (A+C^T)\nabla \rho$. Note that 
$\rho$ is locally bounded below and above on $\R^d$ and $\widehat{\mathbf{F}} \in L_{loc}^p(\R^d, \R^d)$,  $\alpha \rho \psi \in L_{loc}^q(\R^d)$. Let $U$, $V$ be open balls in $\R^d$ satisfying $\overline{U} \subset V$. Since we are integrating over bounded sets and $p>d$, we assume that $p$ is so close to $d$ that $\frac{pd}{p+d}\in (\frac d2,  (\frac1q+\frac1s)^{-1})$. Then, by Theorem \ref{mainregulthm}(ii), there exists a H\"{o}lder continuous $dx$-version $R_{\alpha} f$ of $G_{\alpha} f$ on $\R^d$ and constants $\gamma \in (0,1)$, $c_1>0$, which depend on $U$ and $V$, but not on $f$,  such that 
\begin{eqnarray*}
\|R_{\alpha} f \|_{C^{0,\gamma}(\overline{U})} &\leq& c_1 \big( \|G_{\alpha} f \|_{L^1(V)}+ \| \rho \psi  f\|_{L^{  (\frac{1}{q}+\frac{1}{s})^{-1}}(V)}    \big).
\end{eqnarray*}
Using in particular H\"{o}lder's inequality, the right hand side of the latter inequality can be estimated by
\begin{eqnarray*}
&& c_1 \Big( \Big\| \frac{1}{\rho \psi} \Big\|_{L^{\infty}(V)} \| \rho \psi  G_{\alpha} f\|_{L^1(V)} + \|\rho \psi\|_{L^q(V)} \|f\|_{L^s(V)} \Big)   \\
&\leq&    c_1 \Big( \Big\| \frac{1}{\rho \psi} \Big\|_{L^{\infty}(V)} \| G_{\alpha}f\|_{L^1(V, \widehat{\mu})} + \|\rho \psi\|_{L^q(V)} \Big \|\frac{1}{\rho \psi}  \Big\|^{1/s}_{L^{\infty}(V)} \|f\|_{L^s(V,\widehat{\mu})} \Big)  \\
  &\leq& c_2 \big( \|G_{\alpha} f \|_{L^1(V, \widehat{\mu})}+\| f\|_{L^{s}(V, \widehat{\mu})}    \big),  
\end{eqnarray*}
where $c_2:=c_1\big(\frac{1}{\inf_{V}\rho\psi} \vee  \frac{\|\rho \psi \|_{L^q(V)}}{(\inf_{V}\rho\psi)^{1/s}} \big)$. Thus \eqref{resestko} holds for $f\in L^1(\R^d, \widehat{\mu})_b$. Using the H\"{o}lder inequality and the contraction property, we obtain \eqref{resestko} for $f \in \cup_{r \in [s, \infty)}L^r(\R^d, \widehat{\mu})$.
If $f \in L^{\infty}(\mathbb{R}^d, \widehat{\mu})$, then $1_{B_n} f \in L^1(\mathbb{R}^d, \widehat{\mu})_b$ for all $n \geq 1$ and
$$
\|f-f_n\|_{L^s(V, \widehat{\mu})} + \| G_{\alpha}(f-f_n) \|_{L^1(V,\widehat{\mu})} \to 0 \text{ as } \;n\to \infty
$$ 
by Lebesgue's Theorem and consequently \eqref{resestko} holds.
\end{proof}\\ \\
Assume {\bf (C1)}. Let $U,V,c_2$, and $\gamma$ be as in Theorem \ref{rescondojc} and $g \in  L^r(\R^d,\widehat{\mu})$ for some $r \in [s, \infty]$. Let $\alpha>0$. Then, using Theorem \ref{rescondojc} and the $L^r(\mathbb{R}^d, \widehat{\mu})$-contraction property of $(G_{\alpha})_{\alpha>0}$
\begin{align}
\| R_{\alpha} g\|_{C^{0, \gamma}(\overline{U})} &\leq c_2 \left ( \| G_{\alpha}f \|_{L^1(V,\widehat{\mu})}+ \| f \|_{L^s(V, \widehat{\mu})}\right )  \nonumber \\
& \leq c_2 \left(   \widehat{\mu}(V)^{1-\frac{1}{r}}  \|G_{\alpha}f \|_{L^r(V, \widehat{\mu})}+ \widehat{\mu}(V)^{\frac{1}{s}- \frac{1}{r}}  \|f\|_{L^r(V, \widehat{\mu})}  \right)  \nonumber \\
&\leq c_3  \|f\|_{L^r(\mathbb{R}^d, \widehat{\mu})}, \label{resolcontra}
\end{align}
where  $c_3=c_2\left( \frac{1}{\alpha}  \widehat{\mu}(V)^{1-\frac{1}{r}}  + \widehat{\mu}(V)^{\frac{1}{s}-\frac{1}{r}}   \right)$.
Let $r \in [s, \infty)$ and $f \in  D(L_r)$. Then $f = G_1(1-L_r) f$, hence by Theorem \ref{rescondojc}, $f$ has a locally H\"{o}lder continuous $dx$-version on $\R^d$ and
\begin{eqnarray*}
\|f\|_{C^{0,\gamma}(\overline{U})}
&\leq& c_3 \|f\|_{D(L_r)}.
\end{eqnarray*}
Let $(T_{t})_{t>0}$ be as described in Theorem \ref{helholmop}.
Then, $T_t f \in D(L_r)$ and $T_t f$  has hence a continuous $dx$-version, say $P_{t}f$, with
\begin{equation} \label{ptgraphnorm}
\|P_t f\|_{C^{0, \gamma}(\overline{U})} \leq c_3\|P_t f \|_{D(L_r)}.
\end{equation}
Note that $c_3$ is independent of $t\geq 0$ as well as of $f$. The following Lemma, which can be shown exactly as \cite[Lemma 2.30]{LST22} will be used for the proofs of Theorems \ref{1-3reg3} and \ref{fvodkoko} below and also in Lemma \ref{pjpjwoeij}.
\begin{lemma} \label{contiiokc}
Assume {\bf (C1)}. For any $f\in \bigcup_{r\in [s,\infty)} D(L_r)$, the map
$$
(x,t)\mapsto P_t f(x)
$$
is continuous on $ \R^d\times [0,\infty)$, where $P_0 f(x):=f(x)$ for each $x \in \mathbb{R}^d$.
\end{lemma}

\begin{theorem}\label{1-3reg3}
Assume {\bf (C1)}. Fix $\widetilde{p} \in (d, p]$ with $\frac{\widetilde{p}}{2} \leq q$ and let $ f \in \bigcup_{\nu \in [\frac{2\widetilde{p}}{\widetilde{p}-2}, \infty]}L^\nu(\R^d, \widehat{\mu})$, $t>0$. Then $T_t f$ has a continuous $dx$-version $P_t f$ on $\R^d$ and $P_{\cdot} f(\cdot) $ is continuous on $\R^d \times (0, \infty)$.  Furthermore, for any bounded open sets $U$, $V$ in $\R^d$ with $\overline{U} \subset V$ and $0<\tau_3<\tau_1<\tau_2<\tau_4$, (hence $[\tau_1, \tau_2] \subset (\tau_3, \tau_4)$), we have the following estimate for all $f \in \cup_{\nu \in[\frac{2\widetilde{p}}{\widetilde{p}-2},\infty]} L^\nu(\R^d, \widehat{\mu})$
\begin{equation*} \label{thm main est}
\|P_{\cdot} f(\cdot)\|_{C(\overline{U} \times [\tau_1, \tau_2])} \leq  C_1 \| P_{\cdot} f(\cdot) \|_{L^{\frac{2\widetilde{p}}{\widetilde{p}-2},2}( V \times (\tau_3, \tau_4)) },
\end{equation*}
where $C_1$ is a constant that depends on $\overline{U} \times [\tau_1, \tau_2], V \times (\tau_3, \tau_4)$, but is independent of $f$. 
\end{theorem}
\begin{proof}
First assume $f \in D(\overline{L})_b \cap D(L_s) \cap D(L_2)$. By \cite[Theorem 3.5]{LT24a}, we obtain that $T_t f \in D(\overline{L})_b \subset D(\mathcal{E}^0) \subset H^{1,2}_{loc}(\mathbb{R}^d)$. Moreover, using Lemma \ref{contiiokc} and the sub-Markovian property of $(P_t)_{t>0}$ we see that the function $u$ defined by $u(x,t) := P_t f(x)$, $(x,t) \in \mathbb{R}^d \times [0, \infty)$ is continuous and bounded on $\mathbb{R}^d \times [0, \infty)$. Note that for any bounded open set $O\subset \R^d$ and $T>0$,  it holds $u \in H^{1,2}(O\times (0,T))$ by Lemma \ref{pjpjwoeij}.  Now, for $\varphi_1 \in C_{0}^{\infty}(\R^d)$, $\varphi_{2} \in C_{0}^{\infty}\left((0,T)\right)$, we get that
\begin{align}
&\iint_{\R^d \times (0,T)}\big ( \langle \frac{1}{2} \rho A  \nabla u,  \nabla (\varphi_1 \varphi_2)  \rangle-  \langle  \rho \psi \bold{B}, \nabla u   \rangle  \varphi_1 \varphi_2 \big )\;dx dt  \nonumber \\
=&\int_{0}^{T} \varphi_2\big( \int_{\R^d}  \langle \frac{1}{2} \rho A  \nabla \left( T_t f \right),  \nabla \varphi_1   \rangle-  \big \langle  \rho \psi \bold{B}, \nabla \left(T_t f \right) \big \rangle  \varphi_1 \;dx \big) dt \nonumber \\
=&\int_{0}^{T} \varphi_2 \Big( \mathcal{E}^{0}(T_t f, \varphi_1)- \int_{\R^d} \langle \bold{B}, \nabla T_t f \rangle \varphi_1\, d \widehat{\mu} \Big) dt \nonumber = \int_{0}^{T} - \varphi_2  \big(  \int_{\R^d}  \varphi_1  \overline{L}\, \overline{T}_t f \, d\mu  \big) dt \nonumber \\
=&  \int_{0}^{T}  -\varphi_2 \Big( \frac{d}{dt} \int_{\R^d} \varphi_1 T_t f \;\rho \psi dx  \Big) dt \nonumber =  \int_{0}^{T}  ( \frac{d}{dt} \varphi_2 ) \big (  \int_{\R^d} \varphi_1 T_t f \, \rho\psi dx \big ) dt  \nonumber \\
=&\iint_{\R^d \times (0,T)} u \;\partial_t (  \varphi_1 \varphi_2  ) \rho \psi dx dt. \label{maindivdpl}
\end{align}
By Lemma \ref{stoneweier}, \eqref{maindivdpl} extends to
\begin{eqnarray}
&&\iint_{\R^d \times (0,T)}  \big (\langle \frac{1}{2} \rho A  \nabla u,  \nabla \varphi\rangle-  \langle  \rho \psi \bold{B}, \nabla \left(T_t f \right) \rangle  \varphi \big ) \;dx dt \nonumber  \\
&&\qquad  =\iint_{\R^d \times (0,T)} u \;\partial_t  \varphi  \cdot \rho \psi dx dt \quad \text{ for all } \varphi \in C_0^{\infty}(\R^d \times (0,T)). \label{maindivkjbie-3}
\end{eqnarray}
Analogously to the calculation that led from \eqref{feppwikmee} to \eqref{feppwikmee2}, we obtain 
\begin{align*}
&\iint_{\R^d \times (0,T)} \big ( \big \langle \frac{1}{2} \rho (A+C)  \nabla u,  \nabla \varphi \big \rangle + \langle \widehat{\mathbf{F}}, \nabla u \rangle  \varphi  \big )\;dx dt \nonumber  \\
&\qquad  =\iint_{\R^d \times (0,T)} u \;\partial_t  \varphi  \cdot \rho \psi dx dt \quad \text{ for all } \varphi \in C_0^{\infty}(\R^d \times (0,T)), \label{maindivkjbie-4}
\end{align*}
where $\widehat{\bold{F}} = - \rho \psi \bold{H} +  \frac12 (A+C^T)\nabla \rho\in L^p_{loc}(\mathbb{R}^d, \mathbb{R}^d)$. Let $U,V$ and $\tau_1,\ldots,\tau_4$ be as in the statement of the theorem. Let $\widetilde{q}$ be any number in $[2\vee \frac{\widetilde{p}}{2}, \infty)$ if $q=\infty$, and let $\widetilde{q}:=q$ if $q<\infty$. Then obviously $\psi \in L_{loc}^{\widetilde{q}}(\mathbb{R}^d)$, because $\psi \in L_{loc}^{q}(\mathbb{R}^d)$.  By Lemma \ref{contiiokc} $P_{\cdot }f\in C(\overline{U} \times [\tau_1, \tau_2])$ and then by Theorem \ref{weoifo9iwjeof},  we get for $\nu\in [\frac{2\widetilde{p}}{\widetilde{p}-2},\infty)$
\begin{eqnarray}
\|P_{\cdot }f\|_{C(\overline{U} \times [\tau_1, \tau_2])}\ = \ \|P_{\cdot }f\|_{L^{\infty}(U \times (\tau_1, \tau_2))} &\leq& C_1 \|P_{\cdot} f \|_{L^{\frac{2\widetilde{p}}{\widetilde{p}-2},2}(V \times (\tau_3, \tau_4))}   \label{firstjkcme} \\
&=& C_1 \Big( \int_{\tau_3}^{\tau_4} \big(  \int_{V} |T_t f |^{\frac{2\widetilde{p}}{\widetilde{p}-2}}\, dx \big)^{\frac{\widetilde{p}-2}{\widetilde{p}}}dt   \Big)^{1/2} \nonumber \\
&\leq&  C_1 \big( \frac{1}{ \inf_{V} \rho \psi} \big)^{\frac{\widetilde{p}-2}{2\widetilde{p}}}  \Big( \int_{\tau_3}^{\tau_4} \big(  \int_{V} |T_t f |^{\frac{2\widetilde{p}}{\widetilde{p}-2}} d \widehat{\mu} \big)^{\frac{\widetilde{p}-2}{\widetilde{p}}}dt   \Big)^{1/2} \nonumber   \\
&\leq&  C_1 \underbrace{\big( \frac{1}{ \inf_{V} \rho \psi} \big)^{\frac{\widetilde{p}-2}{2\widetilde{p}}} \widehat{\mu}(V)^{\frac{1}{2}-\frac{1}{\widetilde{p}}-\frac{1}{\nu}} }_{=:C_2} \,\big(\int_{\tau_3}^{\tau_4} \|T_t f \|^2_{L^{\nu}(V, \widehat{\mu})} dt \big)^{1/2} \nonumber \\
&\leq& C_1 C_2 (\tau_4-\tau_3)^{1/2} \|f\|_{L^{\nu}(\R^d, \widehat{\mu})}. \label{conestlmet}
\end{eqnarray}
Now assume $f \in L^1(\R^d, \widehat{\mu})_b$. Then $nG_n f \in D(\overline{L})_b \cap D(L_s) \cap D(L_2)$ for all $ n \in \N$ and $\lim_{n \rightarrow \infty} nG_n f = f$ \; in $L^{\nu}(\R^d, \widehat{\mu})$. Using the completeness of $C(\overline{U} \times [\tau_1, \tau_2])$ and $L^{\frac{2\widetilde{p}}{\widetilde{p}-2},2}(V \times (\tau_3, \tau_4))$, \eqref{firstjkcme} and \eqref{conestlmet} extend to all $f \in L^1(\R^d, \widehat{\mu})_b$. Subsequently, using the same argument, \eqref{firstjkcme} and \eqref{conestlmet} extend to all $f \in L^\nu(\R^d, \widehat{\mu})$, $\nu \in [\frac{2\widetilde{p}}{\widetilde{p}-2}, \infty)$. If  $f \in L^{\infty}(\R^d, \widehat{\mu})$ and let $f_n:= 1_{B_n}\cdot  f$ for $n \geq 1$. Then $\lim_{n \rightarrow \infty} f_n = f$ \, $\widehat{\mu}$-a.e. on $\R^d$ and
\begin{equation}\label{ae2-3}
T_t f= \lim_{n \rightarrow \infty}T_t f_n=\lim_{n \rightarrow \infty}P_t f_n, \; \widehat{\mu} \text{-a.e.\; on } \R^d.  
\end{equation}
 Thus using the sub-Markovian property and then Lebesgue's Theorem in \eqref{firstjkcme}, $(P_{\cdot} f_n (\cdot ))_{n\ge 1}$ is a Cauchy sequence in 
$C(\overline{U} \times [\tau_1, \tau_2])$.
Hence we can again define
$$
P_{\cdot} f:=\lim_{n\to \infty}P_{\cdot} f_n(\cdot) \ \text{ in }\  C(\overline{U} \times [\tau_1, \tau_2]).
$$
For each $t>0$, $P_t f_n$ converges uniformly to $P_t f$ in $U$, hence in view of \eqref{ae2-3}, $T_t f$ has continuous $dx$-version $P_t f$ on $\overline{U}$, $P_{\cdot}f \in  C(\overline{U} \times [\tau_1, \tau_2])$ and \eqref{firstjkcme}   extends to all $f \in L^{\infty}(\R^d, \widehat{\mu})$. Finally, since $U$ and $[\tau_1, \tau_2]$ are arbitrary, we obtain that for any $f\in \cup_{\nu\in[\frac{2\widetilde{p}}{\widetilde{p}-2},\infty]} L^{\nu}(\R^d, \widehat{\mu})$, $P_{\cdot} f(\cdot)$ is continuous on $\R^d \times (0, \infty)$ and for each $t>0$, $P_t f = T_t f$ \,$\widehat{\mu}$-a.e. on $\R^d$.  \vspace{-1.11em}
\end{proof}

\begin{remark}
Assume {\bf (C1)}. Let $(R_{\alpha})_{\alpha>0}$ and $(P_t)_{t>0}$ be as in Theorem \ref{rescondojc}
and
Theorem \ref{1-3reg3}, respectively.
\begin{itemize}
\item[(i)] 
By Theorem \ref{rescondojc}, we get a resolvent kernel $\alpha R_{\alpha}(x, dy)$ for any $\alpha>0$, $x\in \mathbb{R}^d$, defined by
\begin{eqnarray*}\label{resker}
\alpha R_{\alpha}(x,A):=  \alpha R_{\alpha}1_{A}(x), \;\; A\in  \mathcal{B}(\R^d),
\end{eqnarray*}
which is a sub-probability measure on $(\R^d, \mathcal{B}(\R^d))$ and is absolutely continuous with respect to $dx$.
\item[(ii)]
By Theorem \ref{1-3reg3}, we get a semigroup kernel $P_{t}(x,\cdot)$ for any $x\in \R^d$, $t>0$, defined by
\begin{eqnarray*}\label{heatker-3}
P_{t}(x,A):=P_t 1_A(x) , \;\; A\in  \mathcal{B}(\R^d),
\end{eqnarray*}
which is a sub-probability measure on $(\R^d, \mathcal{B}(\R^d))$ and is absolutely continuous with respect to $dx$.
\end{itemize}
\end{remark}
The proof of the following proposition is analogous to the ones of \cite[Propositions 3.1, 3.2 and Theorem 3.3]{LST22}.
\begin{proposition}\label{regular2-3}
Assume {\bf (C1)} and let $t, \alpha>0$. Then it holds:
\begin{itemize}
\item[(i)] $G_{\alpha}g$ has a locally H\"older continuous $dx$-version
\begin{eqnarray}\label{tgyrr-3}
R_{\alpha}g=\int_{\R^d}g(y) R_{\alpha}(\cdot,dy),  \; \; \;  \forall g\in \bigcup_{r\in [s,\infty]} L^r(\R^d,\widehat{\mu}). \qquad 
\end{eqnarray}
In particular,  \eqref{tgyrr-3} extends by linearity  to all $g\in L^s(\R^d, \widehat{\mu})+L^\infty(\R^d, \widehat{\mu})$, i.e.  $(R_{\alpha})_{\alpha>0}$ is $L^{[s,\infty]}(\R^d,\widehat{\mu})$-strong Feller.
\item[(ii)] Let  $\widetilde{p} \in (d, p]$ with $\frac{\widetilde{p}}{2} \leq q$.
$T_t f$ has a continuous $dx$-version   
\begin{eqnarray}\label{tgy2-3}
P_t f= \int_{\R^d} f(y) P_t(\cdot,dy),\; \;\;  \forall f\in \bigcup_{\nu \in [\frac{2\widetilde{p}}{\widetilde{p}-2},\infty]}L^\nu(\R^d, \widehat{\mu}).  \qquad
\end{eqnarray}
In particular,  \eqref{tgy2-3} extends by linearity  to all $f\in L^{\frac{2\widetilde{p}}{\widetilde{p}-2}}(\R^d, \widehat{\mu})+L^\infty(\R^d, \widehat{\mu})$, i.e.   $(P_{t})_{t>0}$ is $L^{[\frac{2\widetilde{p}}{\widetilde{p}-2},\infty]}(\R^d,\widehat{\mu})$-strong Feller.
\end{itemize}
Finally, for any $\alpha>0, x\in \R^d$, $g\in L^s(\R^d, \widehat{\mu})+L^\infty(\R^d, \widehat{\mu})$
$$
R_{\alpha}g(x)=\int_0^{\infty} e^{-\alpha t} P_t g(x)\,dt.
$$
\end{proposition}

\section{Well-posedness} \label{effwegewe}
\subsection{Weak existence}
The following assumption will in particular be necessary to obtain a Hunt process with transition function $(P_t)_{t\ge 0}$ and consequently a weak solution to the corresponding SDE for every starting point. It will be first used in Theorem \ref{existhunt4} below.
\begin{itemize}
\item[\bf (C2)]: {\bf (C1)} holds  and $\mathbf{G}=\frac{1}{2 \psi} \text{div}(A +C)    + \mathbf{H} \in L_{loc}^s(\R^d, \R^d, \widehat{\mu})$ ($s$ is as in {\bf (C1)}, $\widehat{\mu}$ as in Remark \ref{aboutrho}(i)).
\end{itemize}
\begin{remark}
{\bf (C2)} is more general than \cite[{\bf (a), (b)}]{LST22}. Indeed, in the special case where $\psi\equiv 1$, thus $q=\infty$, we may choose $s=\frac{pd}{p+d}$ in {\bf (C2)}. Then \cite[{\bf (a), (b)}]{LST22} is still a stronger assumption and implies  {\bf (C2)}. In particular Theorems \ref{existhunt4}  and \ref{weakexistence4} below are extensions of \cite[Theorem 3.11 and 3.22]{LST22}.
\end{remark}
Let us mention some conditions that imply {\bf (C2)}.
\begin{remark} \label{domgen}
\begin{itemize}
\item[(i)]
Assume  {\bf (C1)} and $ \psi \mathbf{G}  \in L^{s}_{loc}(\R^d,\R^d)$. Then, since $s>\frac{d}{2} \geq 1$, for any bounded open subset $V$ of $\R^d$ it holds
$$
\int_{V}\|\mathbf{G}\|^{s} d \widehat{\mu}  \le \| \rho \|_{L^{\infty}(V)}\int_{V} \|\psi \mathbf{G}\|^s \big |\frac{1}{\psi}\big |^{s-1}
dx \leq  \| \rho \|_{L^{\infty}(V)} \big \| \frac{1}{\psi}\big \|_{L^{\infty}(V)}^{s-1}  \int_{V} \| \psi \mathbf{G} \|^s dx < \infty.
$$
Hence, {\bf (C2)} is satisfied. 
\item[(ii)]
Assume {\bf (C1)} and $\mathbf{G} \in L^{\frac{sq}{q-1}}_{loc}(\R^d,\R^d)$, where $\frac{sq}{q-1}:=s$, if $q=\infty$. Then for any bounded open subset $V$ of $\R^d$, it holds
$$
\int_{V}\|\mathbf{G}\|^{s} d \widehat{\mu} = \int_{V} \| \mathbf{G} \|^s \cdot \rho \psi dx \leq  \|\mathbf{G}\|^s_{L^\frac{sq}{q-1}(V)}\|\rho \psi \|_{L^q(V)}< \infty
$$
Hence, {\bf (C2)} is satisfied. 
\end{itemize}
\end{remark}

\begin{example} \label{forexiste}
Let $d\ge 2$. Let $A=(a_{ij})_{1 \leq i,j \leq d}$ be a symmetric matrix of measurable functions which is (pointwise almost everywhere) locally uniformly strictly elliptic (cf. {\bf (C1)} for the definition). Moreover, assume that 
$a_{ij}, c_{ij} \in VMO_{loc}$ and that $c_{ij}\in L_{loc}^{\infty}(\R^d)$ for each $1 \leq i,j \leq d$. Let $\psi \in L^q_{loc}(\R^d)$ (where $q$ is specified below), $\psi>0$ a.e. on $\mathbb{R}^d$ and $\frac{1}{\psi} \in L_{loc}^{\infty}(\R^d)$.
\begin{itemize}

\item[(i)]
Let $p \in (d, \infty)$, $q \in [p, \infty]$ and $s:=d$. Assume that $\psi \mathbf{H} \in L^p_{loc}(\mathbb{R}^d, \mathbb{R}^d)$, ${\rm div}A  \in L^2_{loc}(\mathbb{R}^d, \mathbb{R}^d, dx)$ and  ${\rm div} (A +C) \in L_{loc}^s(\mathbb{R}^d, \mathbb{R}^d, dx)$. Then, {\bf (C1)} and subsequently {\bf (C2)} is satisfied by Remark \ref{domgen}(i).

\item[(ii)]
Let $p \in (d, \infty)$, $q  \in  [2p, \infty]$ and $s := \frac{2}{3}d$. Assume that $\psi \bold{H} \in L^p_{loc}(\mathbb{R}^d, \mathbb{R}^d)$, ${\rm div}A  \in L^2_{loc}(\mathbb{R}^d, \mathbb{R}^d, dx)$ and ${\rm div} (A +C) \in L_{loc}^{s\vee2}(\mathbb{R}^d, \mathbb{R}^d, dx)$. Then, {\bf (C1)} and subsequently {\bf (C2)} is  satisfied by Remark \ref{domgen}(i).

\item[(iii)]
Let $p \in (d, \infty)$, $q := \infty$ and $s \in (\frac{d}{2}, \infty)$ be arbitrarily chosen, such that $p\ge s \ge 2$. Assume that $\psi \bold{H} \in L^p_{loc}(\mathbb{R}^d, \mathbb{R}^d)$, ${\rm div} A  \in L^2_{loc}(\mathbb{R}^d, \mathbb{R}^d, dx)$ and\, ${\rm div} (A +C) \in L_{loc}^s(\mathbb{R}^d, \mathbb{R}^d, dx)$. Then, {\bf (C1)} and subsequently {\bf (C2)} is  satisfied  by Remark \ref{domgen}(i).

\item[(iv)]

Let $p \in (d, \infty)$, $q \in (\frac{d}{2}, \infty]$, $q\ge 2$, and  $s \in (\frac{d}{2}, \infty)$ be such that $\frac{1}{q}+\frac{1}{s}<\frac{2}{d}$. Assume that $\widehat{\bold{H}} \in L_{loc}^{\frac{sq}{q-1}}(\mathbb{R}^d, \mathbb{R}^d)$,
$\psi \widehat{\bold{H}} \in L^p_{loc}(\mathbb{R}^d, \mathbb{R}^d)$, ${\rm div} A \in L^p_{loc}(\mathbb{R}^d, \mathbb{R}^d)$ and $C=0$. Let $\bold{H}:=\widehat{\bold{H}}-\frac{1}{2\psi} {\rm div} A$. Then, {\bf (C1)} and subsequently {\bf (C2)} is  satisfied by Remark \ref{domgen}(ii).

\item[(v)] In order to present more explicit conditions that imply the conditions of (iv) above let $q \in (d, \infty]$ and $s \in (\frac{d}{2}, \infty)$ satisfy $\frac{1}{q}+\frac{1}{s}<\frac{2}{d}$ and choose $k \in (\frac{dq}{q-d}, \infty]$, i.e. $\frac{1}{k}+\frac{1}{q}<\frac{1}{d}$.  Then choose $p \in (d, (\frac{1}{k}+\frac{1}{q})^{-1} ]$, i.e. $\frac{1}{k}+\frac{1}{q}<\frac{1}{p}<\frac{1}{d}$. If\,  $\widehat{\bold{H}}\in L^{k \vee \frac{sq}{q-1}}_{loc}(\mathbb{R}^d, \mathbb{R}^d)$, ${\rm div} A \in L^p_{loc}(\mathbb{R}^d, \mathbb{R}^d)$, $C=0$ and $\bold{H}:=\widehat{\bold{H}}-\frac{1}{2\psi} {\rm div} A$, then the conditions of (iv) are met, so that {\bf (C1)} and subsequently {\bf (C2)} is  satisfied.

\end{itemize}

\end{example}

\centerline{}
\noindent
Using \cite[Proposition 3.11]{LT24a}, analogously to \cite[Theorem 3.12]{LT18} we obtain:
\begin{theorem}\label{existhunt4}
Under the assumption {\bf (C2)}, there exists a Hunt process
\[
\M =  (\Omega, (\F_t)_{t \ge 0}, \F, (X_t)_{t \ge 0}, (\P_x)_{x \in \R^d\cup \{\Delta\}}   )
\]
with state space $\R^d$ and life time 
$$
\zeta=\inf\{t\ge 0\,:\,X_t=\Delta\}=\inf\{t\ge 0\,:\,X_t\notin \R^d\}, 
$$
having the transition function $(P_t)_{t \ge 0}$ as transition semigroup, such that $\M$ has continuous sample paths in the one point compactification $\R^d_{\Delta}$ of $\R^d$ with the cemetery $\Delta$ as point at infinity, i.e. for any $x \in \R^d$,
$$
\P_{x} \big( \left\{ \omega \in \Omega\, :\, X_{\cdot}(\omega) \in C\big([0, \infty), \R^d_{\Delta}\big) ,\, X_{t}(\omega) = \Delta, \, \forall t \geq \zeta(\omega) \right\} \big) =1.
$$
\end{theorem}
Based on Proposition \ref{regular2-3}(ii) and \eqref{resolcontra} and using a similar method to the one applied in the proof of \cite[Theorem 3.14]{LST22},  the following Krylov-type estimate is derived.
\begin{proposition} \label{resolkrylov}
Assume {\bf (C2)} and let $\M = (\Omega, (\F_t)_{t \ge 0}, \F, (X_t)_{t \ge 0}, (\P_x)_{x \in \R^d\cup \{\Delta\}} )$ be the Hunt process of Theorem \ref{existhunt4}. Let $r \in [s, \infty]$ and $g \in L^r(\mathbb{R}^d, \widehat{\mu})$. 
Then the following Krylov type estimate holds:
for any ball $B$, there exists a constant $C_{B,r}>0$ depending on $B$ and $r$, such that
$$
\sup_{x\in \overline{B}}\E_x\big [ \int_0^t g(X_s) \, ds \big ]  \le   e^t C_{B,r} \|g\|_{L^r(\R^d, \widehat{\mu})} \quad \text{ for all $t \geq 0$.}
$$
\end{proposition}

\noindent
The following theorem can be proved analogously to the proof of \cite[Theorem 3.22]{LST22}.
\begin{theorem}\label{weakexistence4}
Assume {\bf (C2)} is satisfied. Consider the Hunt process $\M$ from Theorem \ref{existhunt4} with coordinates $X_t=(X_t^1,...,X_t^d)$. Let $(\sigma_{ij})_{1 \le i \le d,1\le j \le m}$, $m\in \N$ arbitrary but fixed, be any matrix consisting of measurable functions for all $1 \leq i,j \leq d$, such that $A=\sigma \sigma^T$,  i.e. 
$$
a_{ij}(x)=\sum_{k=1}^m \sigma_{ik}(x) \sigma_{jk}(x), \ \  \text{for } a.e. \  x\in \R^d, \ 1\le i,j \leq d.
$$
Set 
$$
\widehat{\sigma}=\sqrt{\frac{1}{\psi}}\cdot \sigma\text{ , i.e. }\widehat{\sigma}_{ij}=\sqrt{\frac{1}{\psi}}\cdot \sigma_{ij}, \ 1 \leq i \leq d,\ 1 \leq j \leq m.
$$

\begin{itemize}
\item[(i)] Then on a standard extension 
of $(\widetilde{\Omega}, (\widetilde{\mathcal{F}}_t)_{t\ge 0}, \widetilde{\mathcal{F}}, \widetilde{\P}_x )$, $x\in \R^d$, which we denote for notational convenience again 
by $(\Omega, (\mathcal{F}_t)_{t\ge 0}, \mathcal{F}, \P_x )$, $x\in \R^d$, there exists for every $n \in \N$ an $m$-dimensional  standard $(\mathcal{F}_t)_{t \geq 0}$-Brownian motion $(W_{n,t})_{t \geq 0} = \big((W_{n,t}^{1},\dots,W_{n,t}^{m})\big)_{t \geq 0}$ starting from zero such that $\P_x$-a.s. for any $x=(x_1,\ldots,x_d)\in \R^d$, $i=1,\dots,d$
\begin{equation*}
X_t^i = x_i+ \sum_{k=1}^m \int_0^t \widehat{\sigma}_{ik} (X_s) \, dW_{n,s}^{k} +   \int^{t}_{0}   g_i(X_s) \, ds, \quad 0\le  t \leq D_n,
\end{equation*}
where 
$$
D_A:=\inf\{t\ge0\,:\, X_t\in  A\}, \qquad D_n:=D_{\R^d\setminus B_n}, n\ge 1.
$$

Moreover, it holds that $W_{n,s} = W_{n+1, s}$ on $\{s \leq D_n\}$, hence with $W^k_s:=\lim_{n \rightarrow \infty} W^k_{n,s}$, $k=1, \ldots, m$, and $W_s:=(W^1_s, \ldots, W^m_s)$ on $\{ s < \zeta \}$ we get for $1 \leq i \leq d$,
\begin{equation*}
X_t^i = x_i+ \sum_{k=1}^m \int_0^t \widehat{\sigma}_{ij} (X_s) \, dW_s^{k} +   \int^{t}_{0}   g_i(X_s) \, ds, \quad 0\leq  t < \zeta,
\end{equation*}
$\P_x$-a.s. for any $x \in \R^d$. In particular, if $\M$ is non-explosive, then $(W_t)_{t \geq 0}$ is a standard $(\mathcal{F}_t)_{t \geq 0}$-Brownian motion.

\item[(ii)] Suppose that $m=d$ and $\M$ is non-explosive.
Then it holds that $\P_x$-a.s. for any $x=(x_1,\ldots,x_d)\in \R^d$, 
\begin{equation*} \label{itosdeweakglo}
X_t = x+ \int_0^t \widehat{\sigma} (X_s) \, dW_s +   \int^{t}_{0}   \mathbf{G}(X_s) \, ds, \quad 0\le  t <\infty,     
\end{equation*}
i.e. it holds that $\P_x$-a.s. for any $i=1,\ldots,d$
\begin{equation*}\label{weaksolutioneq}
X_t^i = x_i+ \sum_{j=1}^d \int_0^t \widehat{\sigma}_{ij} (X_s) \, dW_s^j +   \int^{t}_{0}   g_i(X_s) \, ds, \quad 0\le  t <\infty,
\end{equation*}
where $W = (W^1,\dots,W^d)$ is a $d$-dimensional standard $(\mathcal{F}_t)_{t \geq 0}$-Brownian motion starting from zero.
\end{itemize}
\end{theorem}
Analogously to the proof of \cite[Corollary 3.27]{LST22}, we obtain the following non-explosion criterion.
\begin{proposition}\label{conservativenesscond}
Assume {\bf (C2)} and \eqref{conscondit}. Then, $\mathbb{M}$ is non-explosive, i.e. $\P_x(\zeta=\infty)=1$ for all $x\in \R^d$ and Theorem \ref{weakexistence4}(ii) applies.
\end{proposition}

\subsection{Well-posedness of degenerate It\^{o}-SDEs}\label{Section 6}
Throughout this section, we consider the following condition.
\begin{itemize}
\item[{\bf (C3)}:]  {\bf (C2)} holds, with $q\in (d+1,\infty]$ and $p=d+1$. Moreover, $q_0 \in [d+1, \infty)$ is such that $\frac{1}{q_0}+\frac{1}{q} = \frac{1}{d+1}$, $\frac{1}{\infty}:=0$, $C\equiv0$, $\text{div} A \in L_{loc}^{d+1}(\mathbb{R}^d, \mathbb{R}^d)$
 and $\mathbf{G}  =\frac{1}{2 \psi} \text{div}A    + \mathbf{H}\in L^{\infty}_{loc}(\R^d,\R^d)$.
\end{itemize}

\begin{proposition}\label{feokokoe1}
Assume {\bf (C3)}. If $u \in D(L_{q_0})$, then $u \in H_{loc}^{2,d+1}(\R^d)$. Moreover, for any open ball $B$ in $\R^d$  there exists a constant $C>0$, independent of $u$, such that 
$$
\|u\|_{H^{2,d+1}(B)} \leq C \|u\|_{D(L_{q_0})}.
$$
\end{proposition}
\begin{proof}
By assumption {\bf (C3)} we get $\psi\mathbf{H}=\psi\mathbf{G}-\frac{1}{2}\text{div} A \in L^{d+1}_{loc}(\R^d, \mathbb{R}^d)$. Thus by Theorem \ref{helholmop} $\rho \in H_{loc}^{1,d+1}(\R^d) \cap C_{loc}^{0, 1-\frac{d}{d+1}}(\R^d)$ and consequently one can check that  $\rho \psi \mathbf{B} \in L^{d+1}_{loc}(\R^d, \mathbb{R}^d)$. Let $f \in C_0^{\infty}(\R^d)$. By \eqref{feppwikmee}
\begin{eqnarray*}
&&\int_{\R^d} \big \langle \frac12 \rho A \nabla G_{1}f, \nabla \varphi \big \rangle dx - \int_{\R^d} \langle \rho \psi \mathbf{B}, \nabla G_{1} f \rangle \varphi\, dx  +\int_{\R^d} (\rho \psi G_{1} f) \,\varphi dx  \nonumber \\
&& = \int_{\R^d} (\rho \psi f) \,\varphi dx, \quad \; \text{ for all } \varphi \in C_0^{\infty}(\R^d). \label{fewbllle}
\end{eqnarray*}
Let $B$, $V$ be open balls in $\mathbb{R}^d$ with $\overline{B} \subset V$. Then, by Proposition \ref{basicpropvm}(iii), Proposition \ref{vmoprop}(ii) and Theorem \ref{mainregulthm}(iv), $G_{1} f \in H^{2,d+1}_{loc}(\mathbb{R}^d)$ and there exists a constant $C_2>0$ (cf. Theorem \ref{mainregulthm}(iv)) which is independent of $f$ such that
\begin{align}
&\|G_{1} f \|_{H^{2,d+1}(B)} \leq C_2 \big( \|G_{1} f \|_{L^1(V)}  + \|\rho \psi f\|_{L^{d+1}(V)} \big)   \nonumber \\
& \leq C_2 \big( dx(V)^{1-\frac{1}{q_0}}  \| (\rho \psi)^{-1} \|^{1/q_0}_{L^{\infty}(V)} \|G_{1} f \|_{L^{q_0}(\mathbb{R}^d, \widehat{\mu})} +\|\rho \psi \|_{L^q(V)} \|(\rho \psi)^{-1} \|^{1/q_0}_{L^{\infty}(V)}
\|f\|_{L^{q_0}(\mathbb{R}^d, \widehat{\mu})}  \big) \nonumber \\
&\leq  C_3 \|f\|_{L^{q_0}(\mathbb{R}^d, \widehat{\mu})}, \label{uniquelq}
\end{align}
where we used the $L^{q_0}(\mathbb{R}^d, \widehat{\mu})$-contraction property and $C_3>0$ is a constant independent of $f$. By approximation, \eqref{uniquelq} extends to all $f \in L^{q_0}(\mathbb{R}^d, \widehat{\mu})$. Let $u \in D(L_{q_0})$ and $g = (1-L_{q_0})u \in L^{q_0}(\mathbb{R}^d, \widehat{\mu})$. Then, $u = G_1 g \in H^{2,d+1}(B)$, hence by \eqref{uniquelq}
$$
\|u\|_{H^{2,d+1}(B)} \leq C_3 \| (1-L_{q_0})u\|_{L^{q_0}(\mathbb{R}^d, \widehat{\mu})} \leq C_3 \|u\|_{D(L_{q_0})}.
$$
\end{proof}

\begin{lemma} \label{pjpjwoeij}
Assume {\bf (C1)}. Let $f \in D(\overline{L})_b \cap D(L_s)  \cap D(L_2)$ and consider as in Lemma \ref{contiiokc} the map defined as
$$
u_f:=P_{\cdot} f \in C_b(\R^d \times [0, \infty))
$$ 
Then for any bounded open set $U$ in $\R^d$, $T>0$, $1 \leq i \leq d$, there exist
$$
\partial_t u_f, \, \partial_i u_f \in L^{2, \infty}(U \times (0,T)),
$$
and for each $t \in (0,T)$, it holds
$$
\partial_t u_f(\cdot,t) = T_t L_2 f \in L^{2}(U), \; \text{ and }\; \partial_i u_f(\cdot, t) = \partial_i P_t f  \in L^2(U).
$$
If {\bf (C3)} holds and  $f \in D(L_{q_0})$ where $q_0$ is as in {\bf (C3)}, then there exists $\partial_i \partial_j u_f  \in L^{d+1, \infty}(U \times (0,T))$  for all $1 \leq i,j \leq d$, and for each $t \in (0, T)$
$$
\partial_i \partial_j u_f (\cdot, t ) = \partial_i \partial_j P_t f \in L^{d+1}(U).
$$
\end{lemma}
\begin{proof}
Assume {\bf (C1)}. Let $f \in D(\overline{L})_b \cap D(L_s)  \cap D(L_2)$ and $t>0$, $t_0 \geq 0$.  Then $P_{t_0} f = \overline{T}_{t_0} f$ a.e., where $\overline{T}_0 := id$ and by \cite[Theorem 3.5(i)]{LT24a}, 
$$
\overline{T}_{t_0} f \in  D(\overline{L})_b \subset D(\mathcal{E}^0).
$$    
Using \eqref{uniellip} and \cite[Theorem 3.5(i)]{LT24a}, for any open ball  $B$ in $\R^d$ with $\overline{U} \subset B$,
\begin{eqnarray}
\| \nabla P_t f - \nabla P_{t_0}f \|^2_{L^2(B)}
&\leq&   2(\lambda_B \inf_{B} \rho)^{-1}\mathcal{E}^0(P_t f - P_{t_0}f, P_t f- P_{t_0}f)  \nonumber \\
&\leq&  2(\lambda_B \inf_{B} \rho)^{-1}  \int_{\R^d} -\overline{L}(\overline{T}_t f- \overline{T}_{t_0}f) \cdot (\overline{T}_t f- \overline{T}_{t_0}f) d \widehat{\mu} \nonumber  \\
&\leq& 4 (\lambda_B \inf_{B} \rho)^{-1}  \|f\|_{L^{\infty}(\R^d, \widehat{\mu})} \|\overline{T}_t\overline{L}f - \overline{T}_{t_0} \overline{L}f \|_{L^1(\R^d, \widehat{\mu})}. \label{feftlklere}
\end{eqnarray}
Likewise, 
\begin{eqnarray*} \label{fekokwe2}
 \| \nabla P_t f \|^2_{L^2(B)} \leq  2(\lambda_B \inf_{B} \rho)^{-1}  \|f\|_{L^{\infty}(\R^d, \widehat{\mu})} \|\overline{T}_t  \overline{L} f\|_{L^1(\R^d, \widehat{\mu})}.
\end{eqnarray*}
For each $i=1, \dots, d$, define the map
$$
\partial_i P_{\cdot} f: [0,T] \rightarrow L^2(U), \; \; \;  t \mapsto \partial_i P_{t} f.
$$
Then by \eqref{feftlklere} and the $L^1(\R^d, \widehat{\mu})$-strong continuity of $(\overline{T}_t)_{t>0}$, the map $\partial_i P_{\cdot} f$ is continuous from $[0,T]$ to $L^2(U)$, hence by \cite[Theorem, p91]{Nev65}, there exists a Borel measurable function $u^i_f$ on $U \times (0,T)$ such that for each $t \in (0,T)$ it holds
$$
u^i_f(\cdot, t) = \partial_i P_t f \in L^2(U).
$$
Using \eqref{feftlklere} and the $L^1(\R^d, \mu)$-contraction property of $(\overline{T}_t)_{t >0}$, we get 
\begin{eqnarray*}
\|u^i_f \|_{L^{2, \infty}(U \times (0,T))} &=& \sup_{t \in (0,T)} \| \partial_i P_t  f\|_{L^2(U)} \\
&\leq&  \sqrt{2}(\lambda_B \inf_{B} \rho)^{-1/2}  \|f\|^{1/2}_{L^{\infty}(\R^d, \widehat{\mu})} \|\overline{L} f\|^{1/2}_{L^1(\R^d,\widehat{\mu})}<\infty.
\end{eqnarray*}
Thus $u^i_f \in L^{2, \infty}(U \times (0,T))$.
Now let $\varphi_1 \in C_0^{\infty}(U)$ and $\varphi_2 \in C_0^{\infty}((0,T))$. Then
\begin{eqnarray}
\iint_{U \times (0,T)} u_f \cdot \partial_i(\varphi_1 \varphi_2) dx dt &=& \int_0^T \big(\int_U P_t f \cdot \partial_i \varphi_1 dx \big) \varphi_2 dt  \nonumber \\
&=& \int_0^T -\big( \int_U \partial_i P_t f  \cdot \varphi_1 dx \big) \varphi_2 dt \nonumber \\
&=& -\iint_U u_f^i  \cdot \varphi_1 \varphi_2 dxdt.   \label{pokpoin3}
\end{eqnarray}
By Lemma \ref{stoneweier},  \eqref{pokpoin3} extends to $\varphi\in C_0^{\infty}(U\times (0,T))$. Consequently $\partial_i u_f = u^i_f \in L^{2,\infty}(U \times (0,T))$.\\
Now define the map
$$
T_{\cdot}L_2 f: [0,T] \rightarrow L^2(U), \; \; \, t \mapsto T_t L_2 f, \quad T_{0}:= id.
$$
Since
\begin{eqnarray*}
\|T_t L_2 f- T_{t_0} L_2 f\|_{L^2(U)} \leq (\inf_{U} \rho \psi)^{-1/2} \|T_t L_2 f- T_{t_0} L_2 f\|_{L^2(\R^d, \widehat{\mu})},
\end{eqnarray*}
using the $L^2(\R^d, \widehat{\mu})$-strong continuity of $(T_t)_{t>0}$ and \cite[Theorem, p91]{Nev65}, there exists a Borel measurable function $u^0_f$ on $U \times (0,T)$ such that for each $t \in (0,T)$ it holds
$$
u^0_f(\cdot, t) = T_t L_2 f \in L^2(U).
$$
Using the $L^2(\R^d, \widehat{\mu})$-contraction property of $(T_t)_{t >0}$,
\begin{eqnarray*}
\|u^0_f \|_{L^{2, \infty}(U \times (0,T))} &=& \sup_{t \in (0,T)} \| T_t L_2 f\|_{L^2(U)} \\
&\leq&  ( \inf_{U} \rho \psi)^{-1/2}  \|L_2 f\|_{L^2(\R^d, \widehat{\mu})} <\infty
\end{eqnarray*}and so $u^0_f \in L^{2, \infty}(U \times (0,T))$. Next, observe that 
\begin{eqnarray*}
\iint_{U \times (0,T)} u_f \cdot \partial_t(\varphi_1 \varphi_2) dx dt &=& \int_0^T \big(\int_U T_t f\cdot   \varphi_1 dx \big) \partial_t\varphi_2 dt \\
&=& \int_0^T -\big( \int_U T_t L_2 f  \cdot \varphi_1 dx \big) \varphi_2 dt,   \\
&=& -\iint_U u_f^0 \cdot\varphi_1 \varphi_2 dxdt.
\end{eqnarray*}
Then, using Lemma \ref{stoneweier} in the same way as above again we obtain $\partial_t u_f = u^0_f \in L^{2, \infty}(U \times (0,T))$. \\
Now assume {\bf (C3)} and let $f \in D(L_{q_0})$, $t>0$, $t_0 \geq 0$.  Then   $P_{t_0} f, P_{t_0} f \in D(L_{q_0})$. Applying Proposition \ref{feokokoe1}, $ P_{t_0} f, P_{t_0} f \in H_{loc}^{2,d+1}(\R^d)$ and for each $ 1 \leq i, j \leq d$, it holds by linearity
\begin{eqnarray*}
 \| \partial_i \partial_j P_t f - \partial_i \partial_j P_{t_0}f \|_{L^{d+1}(U)}  &\leq& \|P_{t} f - P_{t_0} f \|_{H^{2,d+1}(U)} \leq C \| P_t f - P_{t_0} f  \|_{D(L_{q_0})} \nonumber \\
&=& C(\|T_t f- T_{t_0} f \|_{L^{q_0}(\R^d, \widehat{\mu})}  + \|T_t L_{q_0} f - T_{t_0} L_{q_0} f \|_{L^{q_0}(\R^d, \widehat{\mu})}\big ), \label{eorijgoijoi}
\end{eqnarray*}
where $C>0$ is a constant independent of $f$. Similarly
\begin{eqnarray*}
\|\partial_i \partial_j P_{t_0}f \|_{L^{d+1}(U)}  &\leq & C \big(\|T_{t_0} f \|_{L^{q_0}(\R^d, \widehat{\mu})}  + \|T_{t_0} L_{q_0} f \|_{L^{q_0}(\R^d, \widehat{\mu})}\big ).
\end{eqnarray*}
Thus we can define the map
$$
\partial_i \partial_j P_{\cdot} f: [0,T] \rightarrow L^{d+1}(U), \; \;  t \mapsto \partial_i  \partial_j P_{t} f
$$
which is continuous from $[0,T]$ to the $L^{d+1}(U)$. Hence by \cite[Theorem, p91]{Nev65}, there exists a Borel measurable function $u^{ij}_f$ on $U \times (0,T)$ such that for each $t \in (0,T)$,  it holds
$$
u^{ij}_f(\cdot, t) = \partial_i  \partial_j P_t f.
$$
Using Proposition \ref{feokokoe1} and the  $L^{q_0}(\R^d, \mu)$-contraction property of $(T_t)_{t>0}$, 
\begin{eqnarray*}
\|u_f^{ij}\|_{L^{d+1, \infty}(U \times (0,T))}  &\leq& \sup_{t \in (0,T)} \|\partial_i \partial_j P_t f \|_{L^{d+1}(U )} \nonumber \\
&\leq&  \sup_{t \in (0,T)}  \|P_t f\|_{H^{2,d+1}(U)} \leq  \sup_{t \in (0,T)} C \| P_t f \|_{D(L_{q_0})}\nonumber  \nonumber \\
&=& \sup_{t \in (0,T)}   C \left( \|T_t f\|_{L^{q_0}(\R^d, \widehat{\mu})}+ \|T_t L_{q_0} f \|_{L^{q_0}(\R^d, \widehat{\mu})} \right) \nonumber \\
&\leq&  C \|f\|_{D(L_{q_0})}, \label{foooekow}
\end{eqnarray*}
where $C>0$ is a constant which is independent of $f$. Thus $u_f^{ij} \in L^{d+1, \infty}(U \times (0,T))$.
Now analoguously to \eqref{pokpoin3} and to the arguments that follow \eqref{pokpoin3}, we finally get
$$
\partial_i \partial_j u_f = u^{ij}_f \in L^{d+1, \infty}(U \times (0,T)).
$$
\end{proof}

\noindent
The following theorem is a result about existence of regular solutions to the Cauchy problem for Kolmogorov's equation with degenerate diffusion coefficient. It is a result in the analysis of partial differential equations and plays a crucial role in the derivation of  well-posedness below.
\begin{theorem} \label{fvodkoko}
Assume  {\bf (C3)} and  let $f \in C_0^{2}(\R^d) \subset D(L_s)$. Let $(P_t)_{t\ge 0}$ be defined as right before and in Lemma \ref{contiiokc} and set $u_f(x,t) := P_t f(x)$, $(x,t) \in \mathbb{R}^d \times [0, \infty)$. Then 
$$
u_f \in C_b\left(\R^d \times [0, \infty) \right) \cap \big( \displaystyle \bigcap_{r>0} W^{2,1}_{d+1, \infty} (B_r \times (0,\infty)) \big) 
$$ 
satisfying $u_f(x, 0) = f(x)$ for all $x \in \R^d$ such that 
$$
\partial_t u_f \in  L^{\infty}(\R^d \times (0,\infty)), \; \partial_i u_f \in \displaystyle \bigcap_{r>0} L^{\infty}(B_{r} \times (0,\infty)) \; \text{ for all } 1\leq i \leq d,
$$
and
$$
\partial_t u_f = \frac12 \text{\rm trace}\big(\frac{1}{\psi}A \nabla^2 u_f\big)+ \langle \mathbf{G}, \nabla u_f \rangle \, \; \text{ a.e. on } \R^d \times (0, \infty).
$$
\end{theorem}
\begin{proof}
Let $f \in C_0^{\infty}(\R^d)$. Then $f \in D(L_s)$ and by  Lemma \ref{contiiokc}, we get $u_f \in C_b(\R^d \times [0, \infty))$. Moreover, $u_f(x, 0) = f(x)$ for all $x \in \R^d$. Since $\frac{1}{\psi}a_{ij}\in L^{\infty}_{loc}(\R^d)$, $1\le i,j\le d$,  and $\mathbf{G} \in L^{\infty}_{loc}(\R^d, \R^d)$, it holds $f \in D(L_{q_0})$, so that $P_t f \in D(L_{q_0})$ for any $t \geq 0$. By Lemma \ref{pjpjwoeij}, for each $t>0$ it holds $\partial_t u_f(\cdot ,t) = T_t L_s f = T_t L f$\; $\widehat{\mu}$-a.e. on $\R^d$. Note that for each $t>0$, using  the sub-Markovian property,
\begin{eqnarray*}
\| \partial_t u_f(\cdot, t) \|_{L^{\infty}(\R^d)} &=& \|T_t L f \|_{L^{\infty}(\R^d, \widehat{\mu})}\\
&\leq&  \|Lf\|_{L^{\infty}(\R^d, \widehat{\mu})},
\end{eqnarray*}
hence $\partial_t u_f \in L^{\infty}(\R^d \times (0, \infty))$. By Lemma \ref{pjpjwoeij}, for $1 \leq i,j \leq d$, $t>0$, \;
$\partial_i u_f(\cdot, t)  = \partial_i P_t f$,\, $\partial_i \partial_j u_f (\cdot, t)= \partial_i \partial_j P_t f$ \,$\widehat{\mu}$-a.e. on $\R^d$. Using Proposition \ref{feokokoe1} and the $L^{q_0}(\R^d, \widehat{\mu})$-contraction property of $(T_t)_{t>0}$, for any $R>0$ and for each $1 \leq i,j \leq d$, $t>0$, it holds
\begin{eqnarray*}
 \|\partial_i \partial_j u_f(\cdot,t ) \|_{L^{d+1}(B_R )}&\leq& \|P_t f\|_{H^{2,d+1}(B_R)} \leq C \|P_t f \|_{D(L_{q_0})} \nonumber \\
&=& C \left( \|T_t f\|_{L^{q_0}(\R^d, \widehat{\mu})}+ \|T_t L_{q_0} f \|_{L^{q_0}(\R^d, \widehat{\mu})} \right) \nonumber \\
&\leq& C \|f\|_{D(L_{q_0})}, \label{foooekowwev}
\end{eqnarray*}
where $C>0$ is a constant as in Proposition \ref{feokokoe1} and independent of $f$ and $t>0$. By the above and Morrey's inequality, there exists a constant $C_{R,d}>0$, independent of $f$ and $t>0$, such that for each $t>0$, $1 \leq i \leq d$, 
\begin{eqnarray*}
\|\partial_i u_f(\cdot, t)\|_{L^{\infty}(B_R)}&\leq& \|\partial_i P_t f\|_{L^{\infty}(B_R)} \\
&\leq& C_{R,d} \|P_t f\|_{H^{2,d+1}(B_R)}  \\
&\leq& C_{R,d} C \|f\|_{D(L_{q_0})}.
\end{eqnarray*}
Thus, $u_f \in  W^{2,1}_{d+1, \infty} (B_R \times (0,\infty))$ and $\partial_t u_f, \partial_i u_f \in  L^{\infty}(B_{R} \times (0,\infty))$ for all $1 \leq i \leq d$.  By \eqref{maindivkjbie-3}, it holds 
\begin{eqnarray*}
&&\iint_{\R^d \times (0,\infty)}\big ( \langle \frac{1}{2} \rho A  \nabla u_f,  \nabla \varphi  \rangle-   \langle  \rho \psi \bold{B}, \nabla u_f  \rangle  \varphi\big )  \;dx dt \nonumber  \\
&&\qquad  =\iint_{\R^d \times (0,\infty)} \big ( -\partial_t u_f \cdot \varphi  \rho \psi \big )dx dt \quad \text{ for all } \varphi \in C_0^{\infty}(\R^d \times (0,\infty)). 
\end{eqnarray*}
Using integration by parts, we obtain
\begin{eqnarray*}
&&-\iint_{\R^d \times (0,\infty)} \big(\, \frac12 \text{trace}\big( \frac{1}{\psi} A \nabla^2 u_f\big) +
\big \langle \beta^{\rho, A, \psi}+ \bold{B}, \nabla u_f \big  \rangle \big) \varphi  \, d \widehat{\mu}\, dt \nonumber  \\
&&\qquad  =\iint_{\R^d \times (0,\infty)} -\partial_t u_f \cdot \varphi \, d\widehat{\mu}\, dt \quad \text{ for all } \varphi \in C_0^{\infty}(\R^d \times (0,\infty)). 
\end{eqnarray*}
Therefore, since $\bold{G} = \beta^{\rho, A, \psi} + \bold{B}$ (cf. Theorem \ref{helholmop} ), we obtain that 
$$
\partial_t u_f = \frac12 \text{trace}\big( \frac{1}{\psi} A \nabla^2 u_f\big)+ \langle \mathbf{G}, \nabla u_f \rangle \,\, \;\; \text{ a.e. on } \R^d \times (0, \infty).  \vspace{-0.5em}
$$
\end{proof}

\begin{defn}\label{weaksolution}
Suppose {\bf (C2)} holds (for instance if {\bf (C3)} holds). 
Assume that for each open ball $B$ there exist constants $\lambda_B, \Lambda_B>0$ such that \eqref{uniellip} holds. Assume that $\frac{1}{\psi}(x)\ge 0$ for any $x\in \R^d$. Fix $y=(y_1, \ldots, y_d) \in \mathbb{R}^d$ and let
$$
\widetilde{\M}_y=  
\big ((\widetilde{\Omega},  (\widetilde{\F}_t)_{t \ge 0},\widetilde{\F}, \widetilde{\P}_y), (\widetilde{X}_t)_{t \ge 0}, (\widetilde{W}_t)_{t \ge 0}\big )
$$
be such that 
\begin{itemize}
\item[(i)] $(\widetilde{\Omega},  (\widetilde{\F}_t)_{t \ge 0},\widetilde{\F}, \widetilde{\P}_y)$ is a filtered probability space, satisfying the usual conditions,
\item[(ii)] $(\widetilde{X}_t=(\widetilde{X}_t^1,\ldots, \widetilde{X}_t^d))_{t \ge 0}$ is an $(\widetilde{\F}_t)_{t \ge 0}$-adapted continuous $\R^d$-valued stochastic process, 
\item[(iii)] $(\widetilde{W}_t=(\widetilde{W}^1_t,\ldots,\widetilde{W}^d_t))_{t \ge 0}$ is a standard $d$-dimensional  $((\widetilde{\F}_t)_{t \ge 0},\widetilde{\P}_y)$-Brownian motion starting from zero,
\item[(iv)]  for the Borel measurable functions $g_i, \frac{1}{\psi}$, $\widehat{\sigma}_{ij}= \sqrt{\frac{1}{\psi}}\sigma_{ij}$, $1 \leq i,j \leq d$, which are supposed to be chosen  (pointwisely) locally bounded, where $\sigma = (\sigma_{ij})_{1 \leq i,j \leq d}$ is as in Theorem \ref{weakexistence4},  it holds 
$$
\widetilde{\P}_{y}\big (\int_0^t (\widehat{\sigma}_{ij}^2(\widetilde{X}_s)+| g_i(\widetilde{X}_s)|) ds<\infty\big )=1,\ 1\le i,j\le  d,\  t\in [0,\infty),
$$
and for any $1\le i\le d$,
\begin{equation*}
\widetilde{X}_t^i = y_i+ \sum_{j=1}^d \int_0^t \widehat{\sigma}_{ij} (\widetilde{X}_s) \, d\widetilde{W}_s^j +   \int^{t}_{0}   g_i(\widetilde{X}_s) \, ds,
\quad 0\leq t < \infty, \quad \; \text{$\widetilde{\P}_{y}$-a.s.},
\end{equation*}
in short 
\begin{equation} \label{fepojpow}
\widetilde{X}_t = y + \int_0^t \widehat{\sigma}(\widetilde{X}_s) d\widetilde{W}_s + \int_0^t \mathbf{G}(\widetilde{X}_s) ds, \quad 0\leq t < \infty, \quad \; \text{$\widetilde{\P}_{y}$-a.s.},
\end{equation}
where $\widehat{\sigma} = (\widehat{\sigma}_{ij})_{1 \leq i,j \leq d}$.
\end{itemize}
Then $\widetilde{\M}_y$ is called a {\bf weak solution} to \eqref{fepojpow} starting from  $y$.
Note that in this case, $(t, \widetilde{\omega}) \mapsto \widehat{\sigma}(\widetilde{X}_t(\widetilde{\omega}))$ and  $(t, \widetilde{\omega}) \mapsto \mathbf{G}(\widetilde{X}_t(\widetilde{\omega}))$ are (since we assume that the coefficients are chosen Borel measurable and locally bounded) progressively measurable with respect to $(\widetilde{\F}_t)_{t \ge 0}$ and that
$$
\widetilde{D}_R:=\inf \{ t \geq 0 : \widetilde{X}_t \in \R^d \setminus B_R  \}\nearrow \infty \quad \widetilde{\P}_{y}\text{-a.s. }
$$ 
\end{defn}
\begin{remark}\label{remweaksol}
(i) In Definition \ref{weaksolution}  the coefficients of the SDE \eqref{fepojpow} are a priori assumed to be pointwisely locally bounded. This is in particular necessary for the derivation of Theorem \ref{krylovtype} below and all other statements that depend on this theorem. In fact, Theorem \ref{krylovtype} is derived on basis of the classical Krylov type estimate \cite[2. Theorem (2), p. 52]{Kry} for which pointwise local boundedness is needed.\\
(ii) If $\M$ as in Theorem \ref{existhunt4} is non-explosive, then by Theorem \ref{weakexistence4}(ii) there exists a weak solution $\M_y=  \big ((\Omega,  (\F_t)_{t \ge 0},\F, \P_y), (X_t)_{t\ge 0}, (W_t)_{t \ge 0}\big )$ to \eqref{fepojpow} starting from  $y$ for any $y\in \R^d$.     \\
(iii) If $\widetilde{\M}_y$  is a weak solution to \eqref{fepojpow} starting from  $y$, then applying It\^o's formula, we can see that $(\widetilde{X}_t)_{t \ge 0}$ defined on $\big (\widetilde{\Omega},  (\widetilde{\F}_t)_{t \ge 0}, \widetilde{\F}, \widetilde{\P}_y\big )$ solves the {\bf $(L,C_0^{2}(\R^d))$-martingale problem with initial distribution $\delta_y$} (or one also says with initial condition $y$), i.e.
$$
M_t^f(\widetilde{X}):=f(\widetilde{X}_t)-f(y)-\int_0^t Lf(\widetilde{X}_r)dr,\quad t\ge 0,
$$
is an $\big (\widetilde{\Omega},  (\widetilde{\F}_t)_{t \ge 0},\widetilde{\P}_y)\big )$-martingale, with $\widetilde{\P}_y(\widetilde{X}_0=y)=1$. For each $f\in C_0^{2}(\R^d)$, the martingale property is equivalent to 
$$
\widetilde{\E}_y\big [\big (M_t^f(\widetilde{X})-M_s^f(\widetilde{X})\big )\Pi_{i=1}^n h_i(\widetilde{X}_{t_i})\big ]=0,
$$
where $0\le t_1<\ldots <t_n\le s<t$, $h_i\in \mathcal{B}_b(\R^d)$, $1\le i\le n$, $n\in \N$ and therefore only depends on the finite dimensional distributions of $(\widetilde{X}_t)_{t \ge 0}$. In addition, the paths of $(\widetilde{X}_t)_{t \ge 0}$ are obviously continuous. It follows, that when considering the martingale problem, we can always consider a canonical realization $\bar{X}_t(\bar\omega)=\bar \omega(t)$, $t\ge 0$, $\bar\omega\in \bar\Omega= C([0,\infty),\R^d)$, $\bar{\F}_t:=\sigma(\bar{X}_s\,|\, s\in [0,t])$, $\bar{\F}:=\sigma(\bar{X}_s\,|\, s\in [0,\infty))$, $\bar{\P}_y$, of $(\widetilde{X}_t)_{t \ge 0}$, which has the same finite dimensional distributions.
\end{remark}
\begin{lemma}\label{condforpsitohold} 
Assume {\bf (C3)}. Let $y \in \mathbb{R}^d$ and $\widetilde{\M}_y$ be a weak solution to \eqref{fepojpow} starting from  $y$. 
For each $\widetilde{\omega} \in \widetilde{\Omega}$,
let
$$
Z^{\widetilde{\Omega},\widetilde{X}}(\widetilde{\omega}):=\big\{ t \ge 0: \sqrt{\frac{1}{\psi}} (\widetilde{X}_t(\widetilde{\omega})) = 0  \big\}
$$
and
$$
\Lambda(Z^{\widetilde{\Omega},\widetilde{X}}):=  \big\{ \widetilde{\omega} \in \widetilde{\Omega} : dt\big (Z^{\widetilde{\Omega},\widetilde{X}}(\widetilde{\omega})\big ) =0 \big \}.
$$
Consider the following conditions:
\begin{itemize}
\item[(a)] $\widetilde{\P}_y(\Lambda(Z^{\widetilde{\Omega},\widetilde{X}})) =1$.
\item[(b)] (\ref{zerospen}) holds, i.e.
\begin{equation*}
\int_0^{\infty} 1_{\big \{ \sqrt{ \frac{1}{\psi} }=0\big \}}(\widetilde{X}_s)ds=0, \qquad \widetilde{\P}_y\text{-a.s. }
\end{equation*}
\item[(c)]
Let $\bar{\psi}$ be a function on $\mathbb{R}^d$ defined by
$$
\bar{\psi}(x) = \frac{1}{\frac{1}{\psi}(x)} \quad \text{ if \, $\frac{1}{\psi}(x) \in (0, \infty)$} \qquad \text{ and } \quad \bar{\psi}(x) = \infty \quad \text{if \,$\frac{1}{\psi}(x)=0$}.
$$
For each $n\in\N$ and $T>0$,
\begin{equation} \label{kryestintg}
 \widetilde{\E}_y \Big[\int_0^T 1_{B_n}\bar{\psi}(\widetilde{X}_s)ds \Big]<\infty.
\end{equation}
\end{itemize}
Then (a) and (b) are equivalent. Moreover, (c) implies (b).
\end{lemma}
\begin{proof}
Note that \eqref{zerospen} holds, if and only if for $\widetilde{\P}_y$-a.e. $\widetilde{\omega} \in \widetilde{\Omega}$
$$
\widetilde{X}_t(\widetilde{\omega}) \notin \big\{x \in \mathbb{R}^d: \sqrt{\frac{1}{\psi}}(x)  = 0 \big\} \quad \text{ for a.e. $t \in [0, \infty)$}.
$$
Thus, (a) and (b) are equivalent. In order to show that (c) implies (b), let  $n\in\N$ and $T>0$ and assume that 
\eqref{kryestintg} holds. Then, for $\widetilde{\P}_y$-a.e. $\widetilde{\omega} \in \widetilde{\Omega}$
$$
\int_0^T 1_{B_n}\bar{\psi}(\widetilde{X}_s (\widetilde{\omega}))ds <\infty, 
$$
and hence $1_{B_n} \bar{\psi}(\widetilde{X}_t (\widetilde{\omega}))<\infty$ for a.e. $t \in [0, T]$, so that 
 $\bar{\psi}(\widetilde{X}_t (\widetilde{\omega}))<\infty$  for a.e. $t \in [0, T]$. Thus, for $\widetilde{\P}_y$-a.e. $\widetilde{\omega} \in \widetilde{\Omega}$
$$
dt\big( \big\{ t \in [0, T] : \sqrt{\frac{1}{\psi}} (\widetilde{X}_t(\widetilde{\omega}))=0 \big\}  \big)= 0
$$
and so 
$$
\int_0^{T} 1_{\big \{ \sqrt{ \frac{1}{\psi} }=0\big \}}(\widetilde{X}_s)ds=0, \qquad \widetilde{\P}_y\text{-a.s. }
$$
Letting $T \rightarrow \infty$, \eqref{zerospen} follows.
\end{proof}

\begin{theorem}[\bf Local Krylov type estimate] \label{krylovtype}
Assume {\bf (C3)}. Let $y \in \mathbb{R}^d$ and $\widetilde{\M}_y$ be a weak solution to \eqref{fepojpow} starting from  $y$ such that \eqref{zerospen} holds. Let $T>0$, $R>0$ and $f \in L^{d+1}(B_R \times (0,T))$ with $\psi f \in L^{d+1}(B_R \times (0,T))$. 
Then, there exists a constant $C>0$ which is independent of $f$ such that
$$
\widetilde{\E}_{y} \Big[ \int_0^{T\wedge \widetilde{D}_R} f(\widetilde{X}_s,s) ds \Big ] \leq C \| \psi f\|_{L^{d+1}(B_R \times (0,T))},
$$
where $\widetilde{\E}_{y}$ is the expectation with respect to $\widetilde{\P}_{y}$.
\end{theorem}
\begin{proof}
Let $g \in L^{d+1}(B_R \times (0,T))$. Here all functions defined on $B_R \times (0,T)$ are extended to  
$\R^d \times (0, \infty)$ by setting them zero on $\R^d \times (0, \infty) \setminus B_R \times (0,T)$.
Let $Z^{\widetilde{\Omega},\widetilde{X}}$ be defined as in Lemma \ref{condforpsitohold}.  Then, by Lemma \ref{condforpsitohold},  $\widetilde{\P}_y(\Lambda(Z^{\widetilde{\Omega},\widetilde{X}})) =1$ where $\Lambda$ is defined as in Lemma \ref{condforpsitohold}.
Thus, using \cite[2. Theorem (2), p. 52]{Kry},  there exists a constant $C_1>0$ which is independent of $g$, such that
\begin{eqnarray}
&& \widetilde{\E}_y \Big[ \int_{(0, {T \wedge \widetilde{D}_R}) \setminus  Z^{\widetilde{\M}}} \big (2^{-\frac{d}{d+1}}  \text{det}(A)^{\frac{1}{d+1}}\cdot  \Big (\frac{1}{\psi}\Big )^{\frac{d}{d+1}}  g\big )(\widetilde{X}_s,s) ds \Big]  \nonumber \\
&=&\widetilde{\E}_y \Big[ \int_0^{T \wedge \widetilde{D}_R}  \big (2^{-\frac{d}{d+1}}  \text{det}(A)^{\frac{1}{d+1}}\cdot  \Big (\frac{1}{\psi}\big )^{\frac{d}{d+1}}  g \big  )(\widetilde{X}_s,s) ds \Big] \nonumber  \\
&\leq& e^{ T\,  \sup_{x\in B_{R}\|\mathbf{G}(x)\|} } \cdot \widetilde{\E}_y \Big[ \int_0^{T \wedge \widetilde{D}_R} e^{-\int_0^s \|\mathbf{G}(\widetilde{X}_u)\| du}\cdot \text{det}\big(\widehat{A}/2\big)^{\frac{1}{d+1}}g(\widetilde{X}_s,s) ds \Big]\nonumber   \\
&\leq& e^{ T \, \sup_{x\in B_{R}}\|\mathbf{G}(x)\|}   \cdot C_1 \|g\|_{L^{d+1}(B_R \times (0,\infty))}\nonumber   \\
&=& e^{ T \, \sup_{x\in B_{R}}\|\mathbf{G}(x)\|}   \cdot C_1 \|g\|_{L^{d+1}(B_R \times (0,T))}. \label{krylovestim}
\end{eqnarray}
Now define a measurable function $\phi: \mathbb{R}^d \rightarrow \mathbb{R}$ by
$$
\phi(x):= \frac{1}{\frac{1}{\psi}(x)} \quad \text{ if\, $\frac{1}{\psi}(x) \in (0, \infty)$ } \;\; \text{ and } \;\; \phi(x)=0 \quad \text{ if\, $\frac{1}{\psi}(x) =0$.}
$$
Then, $\phi=\psi$ a.e. on $\mathbb{R}^d$. Note that for each $\widetilde{\omega} \in \widetilde{\Omega}$,
$$
Z^{\widetilde{\Omega},\widetilde{X}}(\widetilde{\omega})=\big\{s\ge 0 :  \Big (\frac{1}{\psi}\Big )^{\frac{d}{d+1}} (\widetilde{X}_s(\widetilde{\omega}))\phi^{\frac{d}{d+1}} (\widetilde{X}_s(\widetilde{\omega}))\not= 1 \big \}.
$$
Since \eqref{zerospen} holds, we get from Lemma \ref{condforpsitohold} 
$$
\widetilde{\P}_y\big( \big \{ \widetilde{\omega} \in \widetilde{\Omega}:  dt (\{ s\ge 0\,: \Big (\frac{1}{\psi}\Big )^{\frac{d}{d+1}} (\widetilde{X}_s(\widetilde{\omega}) )\phi^{\frac{d}{d+1}} (\widetilde{X}_s(\widetilde{\omega}))\not= 1\})    =0 \big \} \big)=1.
$$
Let  $f$ be as in the statement of the theorem. Then, replacing $g$ in \eqref{krylovestim} above with $2^{\frac{d}{d+1}} \cdot \text{det}(A)^{-\frac{1}{d+1}} \phi^{\frac{d}{d+1}} f$, we get
\begin{eqnarray*}
&&\widetilde{\E}_{y} \Big[ \int_0^{T\wedge \widetilde{D}_R} f(\widetilde{X}_s,s) ds \Big] =\widetilde{\E}_{y} \Big[ \int_{(0, {T \wedge \widetilde{D}_R}) \setminus Z^{\widetilde{\M}}} f(\widetilde{X}_s,s) ds \Big] \\
&\leq& e^{ T\,  \sup_{x\in B_{R}}\|\mathbf{G}(x)\|}   \cdot C_1 \|2^{\frac{d}{d+1}} \cdot \text{det}(A)^{-\frac{1}{d+1}} \phi^{\frac{d}{d+1}} f\|_{L^{d+1}(B_R \times (0,T))} \\
&=& e^{ T\, \sup_{x\in B_{R}}\|\mathbf{G}(x)\|}    \cdot C_1 \Big \|2^{\frac{d}{d+1}} \cdot \text{det}(A)^{-\frac{1}{d+1}} \left(\frac{1}{\psi}\right)^{\frac{1}{d+1}}  \psi f\Big\|_{L^{d+1}(B_R \times (0,T))} \\
&\leq& \underbrace{2^{\frac{d}{d+1}}e^{ T \, \sup_{x\in B_{R}}\|\mathbf{G}(x)\|}    \cdot C_1 \Big\|\text{det}(A)^{-\frac{1}{d+1}} \left(\frac{1}{\psi}\right)^{\frac{1}{d+1}} \Big \|_{L^{\infty}(B_R)}}_{=:C}\, \|\psi f\|_{L^{d+1}(B_R \times (0,T))}.\\
\end{eqnarray*}
\end{proof}
\begin{remark}\label{casemdandindepbor}
(i) For any weak solution $\widetilde{\M}_y$ starting from $y$ such that \eqref{zerospen} holds, Theorem \ref{krylovtype} implies that  integrals of the form $\int_0^t f(\widetilde{X}_s,s)ds$ \,are, whenever they are well-defined, $\widetilde{\P}_y$-a.s. independent of the particular Borel version that is chosen for $f$. In particular, as an immediate consequence of Proposition \ref{resolkrylov}, the weak solution $\M_y$ constructed in Theorem \ref{weakexistence4}(ii) (cf. Remark \ref{remweaksol}(ii)) satisfies \eqref{kryestintg} for all $T>0$ and $n \in \mathbb{N}$, and \eqref{zerospen}, for any $y\in \R^d$. \\
(ii) In Definition \ref{weaksolution}, we can also consider the case $m\not=d$ as it is considered in Theorem \ref{weakexistence4}(i). Then Theorem \ref{krylovtype} can be shown exactly as it is for $m\ge d$ (here $m<d$ is excluded as the estimate in \cite[2. Theorem (2), p. 52]{Kry} is only developped for $m\ge d$). By this it is then also possible to obtain uniqueness in law as in Theorem \ref{weifjowej} in the case $m\ge d$.
\end{remark}

\begin{theorem}[\bf Local It\^{o}-formula]\label{fefere}
Assume {\bf (C3)}. Let $y \in \mathbb{R}^d$ and $\widetilde{\M}_y$ be a weak solution to \eqref{fepojpow} starting from $y$ such that \eqref{zerospen} holds. Let $R_0>0$ with $y \in B_{R_0}$ and $T>0$. Assume that $u \in W^{2,1}_{d+1}(B_{R_0} \times (0,T)) \cap C(\overline{B}_{R_0} \times [0, T] )$ such that $\partial_t u \in L^{\infty}(B_{R_0} \times (0,T))$ and
$\|\nabla u\| \in L^{\infty}(B_{R_0} \times (0,T))$. Then for $y \in B_R$ where $R \in (0, R_0)$, it holds $\widetilde{\P}_{y}$-a.s.
$$
u(\widetilde{X}_{T\wedge \widetilde{D}_R}, T \wedge \widetilde{D}_R) - u(y,0) = \int_0^{ T \wedge \widetilde{D}_R} \nabla u(\widetilde{X}_s,s) \widehat{\sigma}(\widetilde{X}_s) d\widetilde{W}_s+
\int_0^{ T\wedge \widetilde{D}_R} (\partial_t u + Lu) (\widetilde{X}_s,s) ds, 
$$
where $Lu:= \frac12 {\rm trace}(\widehat{A} \nabla^2 u)+ \langle \mathbf{G}, \nabla u \rangle$
and $\widehat{\sigma} =(\widehat{\sigma}_{ij})_{1 \leq i,j \leq d}$ is a $d\times d$ matrix of functions as in Definition \ref{weaksolution}.

\end{theorem}
\begin{proof}
Take $T_0>0$ satisfying $T_0>T$. Extend $u$ to $\overline{B}_{R_0} \times [-T_0,T_0]$ by
$$
u(x, t)  = u(x, 0) \; \text{ for } -T_0 \leq t<0, \quad  u(x, t)  = u(x, T) \; \text{ for } T<t \leq T_0, \, x \in \overline{B}_{R_0}.
$$
Then it holds 
$$
u  \in W^{2,1}_{d+1}(B_{R_0} \times (0,T)) \cap C(\overline{B}_{R_0} \times [-T, T] ) \ \text{ and } \ \partial_t u, \|\nabla u\| 
\in L^{r}(B_{R_0} \times (-T_0,T_0)), r\in [1,\infty).
$$
Let $\zeta$ be a standard mollifier on $\R^{d+1}$ and consider for $n\in \N$, $\zeta_{n}(x,t):=\frac{1}{(n+N_0)^{d+1}}\zeta((n+N_0)(x,t))$, for sufficiently large $N_0\in \N$. Then it holds $u_n:= u * \zeta_n \in C^{\infty}(\overline{B}_R \times [0, T])$, such that $\lim_{n \rightarrow \infty} \|u_n-u\|_{W^{2,1}_{d+1}(B_R \times (0,T))}=0$, $\lim_{n \rightarrow \infty} \|\nabla u_n-\nabla u\|_{L^{r}(B_R \times (0,T))}=0$, and  $\| \partial_t u - \partial_t u_n \|_{L^{r}(B_R \times (0,T))}$, for all $r\in [1,\infty)$. Let $y \in B_R$. By It\^{o}'s formula,  for any $n\geq 1$
\begin{eqnarray} 
&&u_n(\widetilde{X}_{T\wedge \widetilde{D}_R}, T \wedge \widetilde{D}_R) - u_n(y,0) \nonumber \\[5pt]
&=& \int_0^{ T \wedge \widetilde{D}_R} \nabla u_n(\widetilde{X}_s,s)\,\widehat{\sigma}(\widetilde{X}_s) d\widetilde{W}_s+\int_0^{ T\wedge \widetilde{D}_R}(\partial_t u_n + Lu_n) (\widetilde{X}_s,s) ds,  \;\; \;\; \widetilde{\P}_{y}\text{-a.s. } \qquad \;\; \label{itoaprox}
\end{eqnarray}
By the Sobolev embedding, there exists a constant $C>0$, independent of all $u_n$ and $u$, such that
$$
\displaystyle \sup_{\overline{B}_R \times [0,T]} |u_n-u| \leq  C\|u_n-u\|_{W^{1,1}_{d+1}}(B_R \times (0,T)).
$$
Thus $\lim_{n \rightarrow \infty} u_n(y,0) = u(y,0)$ and $\P_y$-a.s.
$$
u_n(\widetilde{X}_{ T \wedge \widetilde{D}_R}, T \wedge \widetilde{D}_R)  \longrightarrow u(\widetilde{X}_{ T \wedge \widetilde{D}_R}, T \wedge \widetilde{D}_R) \, \text{ as }\, n \rightarrow \infty.
$$ 
By Theorem \ref{krylovtype},
\begin{eqnarray*} 
&&\widetilde{\mathbb{E}}_y\Big[ \big| \int_0^{T\wedge \widetilde{D}_R} (\partial_t u_n + Lu_n) (\widetilde{X}_s,s)ds  -\int_0^{T\wedge \widetilde{D}_R} (\partial_t u + Lu) (\widetilde{X}_s,s)ds  \big |\Big] \\
&&\leq \widetilde{\E}_y \Big [\int_0^{T \wedge \widetilde{D}_R} |\partial_t u-\partial_t u_n|(\widetilde{X}_s,s) ds \Big ] + \widetilde{\E}_y \Big [ \int_0^{T \wedge \widetilde{D}_R} |Lu-Lu_n|(\widetilde{X}_s,s)ds   \Big ] \\
&&\leq C\| \psi \partial_t u- \psi \partial_t u_n\|_{L^{d+1}(B_R \times (0,T))} + C\|\psi Lu-\psi Lu_n\|_{L^{d+1}(B_R \times (0,T))}  \\
&& \leq C \|\psi \|_{L^q(B_R \times (0,T))} \| \partial_t u - \partial_t u_n \|_{L^{\left( \frac{1}{d+1}-\frac{1}{q}  \right)^{-1}}(B_R \times (0,T))} \\
&& \quad \quad +C \sum_{i,j=1}^d  \|a_{ij}\|_{L^{\infty}(B_R)} \cdot \| \partial_{ij} u - \partial_{ij} u_n \|_{L^{d+1}(B_R \times (0,T))}
\\
&& \quad \quad + C \sum_{i=1}^d \|g_i \psi \|_{L^{q}(B_R)}  \| \partial_i u - \partial_i u_n \|_{L^{\left( \frac{1}{d+1}-\frac{1}{q}  \right)^{-1}}(B_R \times (0,T))}  \\
&& \; \longrightarrow 0 \quad \; \text{as } n \rightarrow \infty, \;
\end{eqnarray*}
where $C>0$ is a constant which is independent of $u$ and $u_n$. Using Jensen's inequality, It\^{o}-isometry and Theorem \ref{krylovtype}, we obtain with $|\sigma|_F:=\sqrt{\sum_{i,j=1}^d \sigma_{ij}^2}$
\begin{eqnarray*}
&&\widetilde{\E}_y \Big[\big| \int_0^{ T \wedge \widetilde{D}_R} \big ( \nabla u_n(\widetilde{X}_s,s) - \nabla u(\widetilde{X}_s,s) \big )\,\widehat{\sigma}(\widetilde{X}_s) dW_s \big| \Big] \\
&&\leq  \widetilde{\E}_y \Big [ \big| \int_0^{ T \wedge \widetilde{D}_R} \big ( \nabla u_n(\widetilde{X}_s,s) - \nabla u(\widetilde{X}_s,s) \big )\,\widehat{\sigma}(\widetilde{X}_s) dW_s \big|^2 \Big]^{1/2} \\
&&= \widetilde{\E}_y \Big[ \int_0^{ T \wedge \widetilde{D}_R}\big \|  \big ( \nabla u_n(\widetilde{X}_s,s) - \nabla u(\widetilde{X}_s,s) \big )\,\widehat{\sigma} (\widetilde{X}_s) \big \| ^2 ds\Big]^{1/2}  \\
&& \leq  \big(C \big \| \|(\nabla u_n - \nabla u) \widehat{\sigma} \|^2 \psi  \big \|_{L^{d+1}(B_R \times (0,T))} \big)^{1/2}\\
&&\leq \sqrt{C}\| (\nabla u_n - \nabla u) \sigma \|_{L^{2d+2}(B_R \times (0,T))} \\
&&\leq \sqrt{C}\||\sigma|_F \|_{L^{\infty}(B_R)} \| \nabla u_n- \nabla u\|_{L^{2d+2}(B_R \times (0,T))} \longrightarrow 0 \; \text{ as } n \rightarrow \infty.
\end{eqnarray*}
Letting $n \rightarrow \infty$ in \eqref{itoaprox}, the assertion holds.
\end{proof}

\begin{defn}\label{wellposedness}
Let $y \in \mathbb{R}^d$.
We say that {\bf uniqueness in law} holds for \eqref{fepojpow} starting from $y$, if 
for any two weak solutions 
$$
\widetilde{\M}_y =  
\big ((\widetilde{\Omega},  (\widetilde{\F}_t)_{t \ge 0},\widetilde{\F}, \widetilde{\P}_y), (\widetilde{X}_t)_{t \ge 0}, (\widetilde{W}_t)_{t \ge 0}\big )
$$ 
and 
$$
\widehat{\M}_y=  
\big ((\widehat{\Omega}, (\widehat{\F}_t)_{t \ge 0}, \widehat{\F}, \widehat{\P}_y), (\widehat{X}_t)_{t \ge 0}, (\widehat{W}_t)_{t \ge 0}\big )
$$ 
of  \eqref{fepojpow} starting from $y$, it holds $\widetilde{\P}_y\circ \widetilde{X}^{-1}=\widehat{\P}_y\circ \widehat{X}^{-1}$ on $\mathcal{B}(C([0, \infty), \mathbb{R}^d))$. 
\end{defn}
\begin{theorem} \label{weifjowej}
Assume {\bf (C3)} and that $\M =  (\Omega, (\F_t)_{t \ge 0}, \F, (X_t)_{t \ge 0}, (\P_x)_{x \in \R^d\cup \{\Delta\}}   )$ as in Theorem \ref{existhunt4} is non-explosive. Let $y \in \mathbb{R}^d$ be arbitrarily given but fixed and 
$$
\widetilde{\M}_y =  
\big ((\widetilde{\Omega},  (\widetilde{\F}_t)_{t \ge 0},\widetilde{\F}, \widetilde{\P}_y), (\widetilde{X}_t)_{t \ge 0}, (\widetilde{W}_t)_{t \ge 0}\big )
$$
be a weak solution to \eqref{fepojpow} starting from $y$ such that \eqref{zerospen} holds. Then
$$
\widetilde{\P}_y \circ \widetilde{X}^{-1} = \P_y \circ X^{-1} \quad \text{ on\, $\mathcal{B}(C([0, \infty), \mathbb{R}^d))$}.
$$
In particular, uniqueness in law holds for \eqref{fepojpow} starting from $y$ among all weak solutions $\widetilde{\M}_y$ satisfying \eqref{zerospen}. 
\end{theorem}
\begin{proof} 
Let $f \in C_0^{\infty}(\R^d)$. For  $T>0$, define $g(x,t):=u_f(x, T-t)$,  $(x,t) \in \R^d \times [0,T]$, where $u_f$ is defined as in Theorem \ref{fvodkoko}.
Then by Theorem \ref{fvodkoko},
\begin{eqnarray*}
&&g \in C_b\left(\R^d \times [0, T] \right) \cap \big( \displaystyle \bigcap_{r>0} W^{2,1}_{d+1, \infty} (B_r \times (0,T)) \big), \\
&&\partial_t g \in  L^{\infty}(\R^d \times (0,T)),\; \,  \; \partial_i g \in \displaystyle \bigcap_{r>0} L^{\infty}(B_{r} \times (0,T)),  \, 1\leq i \leq d,
\end{eqnarray*}
and it holds
\begin{eqnarray*}
\frac{\partial g}{\partial t} + Lg = 0 \;\;\,  \text{ a.e. in } \R^d \times (0,T), \;\; \; \;g(x, T) = f(x) \;\; \text{for all }  x \in \R^d.
\end{eqnarray*}
Let $R>0$ be such that $y\in B_R$ and set $\widetilde{D}_R:=\inf \{ t \geq 0 : \widetilde{X}_t\in \R^d \setminus B_R  \}$.
Applying Theorem \ref{krylovtype} to $\widetilde{\M}_y$, it holds 
$$
\widetilde{\E}_y \Big[ \int_0^{T \wedge \widetilde{D}_R}  \Big \vert \frac{\partial g}{\partial t} + Lg \Big \vert (\widetilde{X}_s,s) ds   \Big] =0, \; \; \quad
$$
hence
$$
\int_0^{T \wedge \widetilde{D}_R}  \big( \frac{\partial g}{\partial t} + Lg  \big)(\widetilde{X}_s,s)\,ds  = 0, \;\;\;  \widetilde{\P}_y\text{-a.s., }
$$
and thus by Theorem \ref{fefere},
$$
g(\widetilde{X}_{T \wedge \widetilde{D}_R}, T\wedge \widetilde{D}_R) - g(x,0) = \int_0^{T \wedge \widetilde{D}_R} \nabla g(\widetilde{X}_s,s) \widehat{\sigma}(\widetilde{X}_s) d\widetilde{W}_s, \; \;\; \; \widetilde{\P}_{y}\text{-a.s. }
$$
Since $(\nabla g)\widehat{\sigma}$ is a.e. bounded on $\overline{B}_R\times [0,T]$, by Remark \ref{casemdandindepbor}(i), we can see that $\big (g(\widetilde{X}_{t \wedge \widetilde{D}_R},t\wedge \widetilde{D}_R) - g(x,0)\big )_{t\ge 0}$ is a martingale starting from $0$. Thus
$$
\widetilde{\E}_y \left[ g(\widetilde{X}_{T\wedge \widetilde{D}_R},T \wedge \widetilde{D}_R)  \right]  = g(y,0).
$$
Letting $R \rightarrow \infty$ and using Lebesgue's Theorem, we obtain 
$$
\widetilde{\E}_y[f(\widetilde{X}_T)] =\widetilde{\E}_y[g(\widetilde{X}_T, T)]  = g(y,0).
$$
Since $\M$ as in Theorem \ref{existhunt4} is non-explosive, by Remark \ref{remweaksol}(ii), 
$$
\M_x=  \big ((\Omega,  (\F_t)_{t \ge 0},\F, \P_x), (X_t)_{t\ge 0}, (W_t)_{t \ge 0}\big )
$$  
is a weak solution to \eqref{fepojpow} starting from $x$ for any $x\in \R^d$.
Thus, analogously for $\M_y$, we obtain $\E_y[f(X_T)] = g(y,0)$ and so
$$
\widetilde{\E}_y[f(\widetilde{X}_T)]=\E_y[f(X_T)].
$$ 
It follows that the one-dimensional marginal distributions of $(\widetilde{X}_t)_{t \ge 0}$ with respect to $\widetilde{\P}_y$ and $(X_t)_{t \ge 0}$ with respect to $\P_y$ and coincide, i.e. $\widetilde{\P}_y \circ \widetilde{X}_t^{-1}=\P_y \circ X_t^{-1}$, $t\ge 0$. By Remark \ref{remweaksol}(iii), $(X_t)_{t\ge 0}$ solves the $(L,C_0^{2}(\R^d))$-martingale problem with initial condition $x$ for any $x\in \R^d$ and  $(\widetilde{X}_t)_{t \ge 0}$ solves the $(L,C_0^{2}(\R^d))$-martingale problem with initial condition $y$ and we may hence consider their canonical realizations $\big (\bar\Omega,  (\bar{\F}_t)_{t\ge 0}, \bar{\F}, (\bar{X}_t)_{t\ge 0}, \bar{\P}_x\big )$, $x\in \R^d$ and $\big (\bar\Omega,  (\bar{\F}_t)_{t\ge 0}, \bar{\F}, (\bar{X}_t)_{t\ge 0}, \bar{\widetilde{\P}}_y\big )$. Now define 
$$
\bar\Q_x:=\bar{\P}_x, \quad x\in \R^d\setminus \{y\}, \qquad\bar\Q_y:=\bar{\widetilde{\P}}_y.
$$
Then both $((\bar{X}_t)_{t\ge 0},\bar{\P}_x)$ and  $((\bar{X}_t)_{t\ge 0},\bar{\Q}_x)$ are solutions to the $(L,C_0^{2}(\R^d))$-martingale problem with initial condition $x$ for any $x\in \R^d$ (and thus also for any initial distribution $\nu$) and both have the same one dimensional marginal distributions for any initial distribution  $\nu$, i.e. on $\mathcal{B}(C([0, \infty), \mathbb{R}^d))$
$$
\bar{\P}_\nu(\bar{X}_t\in \cdot)=\int_{\R^d}\bar{\P}_x(\bar{X}_t\in \cdot)\nu(dx)=\int_{\R^d}\bar{\Q}_x(\bar{X}_t\in \cdot)\nu(dx)=\bar{\Q}_\nu(\bar{X}_t\in \cdot), \quad t\ge 0.
$$
Then exactly as in the proof of \cite[Chapter 4, Theorem 4.2]{EthKurtz86}, we can see that $(\bar X_t)_{t\ge 0}$ satisfies the Markov property with respect to $\bar{\Q}_x$  and $\bar{\P}_x$  for each initial condition $x\in \R^d$, hence in particular for the initial condition $y$. Then, since the one dimensional distributions of $(\bar X_t)_{t\ge 0}$ with respect to $\bar{\Q}_y$  and $\bar{\P}_y$ coincide, by the Markov property, also the finite dimensional distributions of $(\bar X_t)_{t\ge 0}$ with respect to $\bar{\Q}_y$  and $\bar{\P}_y$ coincide. Thus we get
$$
\bar{\P}_y \circ \bar{X}^{-1}=\bar\Q_y \circ \bar{X}^{-1} \qquad \text{on }\quad \mathcal{B}(C([0, \infty), \mathbb{R}^d)).
$$
Since the canonical models have the same finite dimensional distributions than their original counterparts, the latter implies
$$
\P_y \circ X^{-1}=\widetilde{\P}_y \circ \widetilde{X}^{-1} \qquad \text{on }\quad \mathcal{B}(C([0, \infty), \mathbb{R}^d)).
$$
\end{proof}

\begin{example}\label{scopepsi}
In the following example, we illustrate the scope of our weak existence and uniqueness results for an explicitly chosen $\psi$. Let $p \in (d, \infty)$ be arbitrarily given and $A=(a_{ij})_{1 \leq i,j \leq d}$ be a symmetric matrix of measurable functions satisfying \eqref{uniellipa.e.} for each open ball $B$, ${\rm{div}} A \in L^p_{loc}(\mathbb{R}^d, \mathbb{R}^d)$ and $a_{ij} \in VMO_{loc}$ for all $1 \leq i,j \leq d$. Let $\sigma=(\sigma_{ij})_{1 \leq i,j \leq d}$ be a matrix of locally bounded and measurable functions on $\mathbb{R}^d$ with $A=\sigma \sigma^T$ pointwise on $\mathbb{R}^d$.
Let $\widehat{\mathbf{H}} \in L^p_{loc}(\mathbb{R}^d,\mathbb{R}^d)$ and $\phi \in L^{\infty}_{loc}(\mathbb{R}^d)$ and assume that for each open ball $B$ in $\mathbb{R}^d$ there exist constants $c_B>0, C_B>0$ such that
\begin{equation} \label{condiforphi}
c_B \leq \phi(x) \leq C_B \quad \text{ for a.e. $x \in B$}.
\end{equation}
For some $\alpha>0$, let $\psi$ be a measurable function such that $\psi(x)=\frac{1}{\phi(x)} \cdot \frac{1}{\|x\|^{\alpha}}$ for all $x \in \mathbb{R}^d \setminus \{0\}$. Consider the following conditions:
\begin{itemize}
\item[(a)]
$\alpha \cdot (\frac{p}{2} \vee 2)<d$, \, $\widehat{\bold{H}} \equiv 0$ on $B_{\varepsilon}(0)$ for some $\varepsilon>0$ and $\widehat{\bold{H}} \in L_{loc}^r(\mathbb{R}^d, \mathbb{R}^d)$ where $r \geq p$ with 
$(\frac{p}{2} \vee 2)^{-1}+\frac{1}{r}<\frac{2}{d}$,
\item[(b)]
$\alpha p<d$, \,$\widehat{\bold{H}}\in L^{\infty}(B_{\varepsilon}(0)) $ for some $\varepsilon>0$,
\item[(c)]
$2\alpha p<d$, $\widehat{\bold{H}}\in L^{2p}(B_{\varepsilon}(0)) $ for some $\varepsilon>0$.
\end{itemize}
Then, either of the conditions (a), (b), or (c) imply {\bf (C2)} where $C\equiv 0$, $\bold{H}=\widehat{\bold{H}}-\frac{1}{2\psi}{\rm  div} A$.
Indeed, for each case (a), (b), or (c),  $q$ is chosen to be any numbers
$q \in [\frac{p}{2}\vee 2, \frac{d}{\alpha})$, $q \in [p, \frac{d}{\alpha})$, or $q \in [2p, \frac{d}{\alpha})$, and
$s$ is chosen to be $r$, $d$, or $\frac{2d}{3}$, respectively. If (a), (b), or (c) holds, assume that $\mathbb{M}$ is non-explosive (for instance if \eqref{conscondit} holds). Then, by Theorem \ref{weakexistence4}(ii), there exists a (strong Markov) weak solution to
\begin{equation} \label{oursdexample}
X_t = x+ \int_0^t \big (\|X_s\|^{\frac{\alpha}{2}} \sqrt{\phi(X_s)}+ \gamma 1_{\{0\}}(X_s)\big )\,
\sigma(X_s) \, dW_s +   \int^{t}_{0}   \widehat{\mathbf{H}}(X_s) \, ds, \quad 0\le  t <\infty,
\end{equation}
for each $x \in \mathbb{R}^d$, $\gamma\ge 0$. Note that the integrals in \eqref{oursdexample} are up to indistinguishability the same no matter which Borel versions are chosen for the dispersion and drift coefficients (cf. Remark \ref{casemdandindepbor}(i)). Now, in order to discuss uniquenessin law, assume the following condition:
\begin{itemize}
\item[(d)]
$\alpha(d+1)<d$, ${\rm{div}} A\in L^{d+1}_{loc}(\mathbb{R}^d, \mathbb{R}^d)$, $\widehat{\bold{H}}$ is (pointwisely) locally bounded and for each open ball $B$, \eqref{uniellip} and $c_B \leq \phi(x) \leq C_B$ for all $x \in B$, for some constants $c_B>0, C_B>0$, i.e. \eqref{condiforphi} is fulfilled where \lq\lq for a.e.\rq\rq \, is replaced by \lq\lq for all\rq\rq, $\sqrt{\frac{1}{\psi}(x)}=\|x\|^{\frac{\alpha}{2}}\sqrt{\phi(x)} + \gamma 1_{\{0\}}(x)$, $x\in \R^d$ for some $\gamma\ge 0$.
\end{itemize}
Then (b) holds with $p=d+1$ and so {\bf (C3)} holds with $C\equiv 0$, $\bold{H}=\widehat{\bold{H}}-\frac{1}{2\psi}{\rm  div} A$, and $q \in (d+1, \frac{d}{\alpha})$ and subsequently we can see that {\bf (C)} is satisfied. Thus by Theorem \ref{weifjowej} applied to a fixed $y\in \R^d$, uniqueness in law for \eqref{oursdexample} starting from $y$ holds among all solutions $\widetilde{X}$ starting from $y$ that spend zero  time with respect to Lebesgue measure at the zeros of $\frac{1}{\psi}$ (i.e. \eqref{zerospen} holds).\\
Of course the condition \eqref{zerospen} is always fulfilled if 
$\|x\|^{\frac{\alpha}{2}}\sqrt{\phi(x)} + \gamma 1_{\{0\}}(x)>0$ 
for all $x\in \R^d$, i.e. if $\gamma>0$. In particular, any two solutions starting from $y$ with possibly different strictly positive $\gamma$ have the same law and satisfy the Markov property. However, two solutions, one corresponding to $\gamma=0$ and the other corresponding to $\gamma>0$ can be different, if the one that corresponds to $\gamma=0$ spends a strictly positive amount of time at zero (cf. Remark \ref{remex} below).
\end{example}

\begin{remark}\label{remex}
As a special case of \eqref{oursdexample}, consider a generalization of the Girsanov SDE to dimensions $d\ge 2$, namely 
\begin{eqnarray}\label{generalizedGirsanov}
X_t=x+ \int_0^t \|X_s\|^{\frac{\alpha}{2}}dW_s+\int^{t}_{0}   \widehat{\mathbf{H}}(X_s) \, ds, \quad 0 \leq t <\infty,
\end{eqnarray}
where $\widehat{\mathbf{H}}$ is locally bounded with at most linear growth and $\widehat{\mathbf{H}}(0)=0$. If $x=0$, then $X \equiv 0$ is a solution. On the other hand, we know from Theorem \ref{weakexistence4}(ii)  that there exists also a solution that spends (Lebesgue measure) zero time at $0$ for all $x \in \mathbb{R}^d$ (cf. Remark \ref{casemdandindepbor}(i)).
Thus uniqueness in law may fail if condition \eqref{zerospen} is not imposed. This is well-known in dimension one where uniqueness in law fails for the Girsanov SDE (i.e. \eqref{generalizedGirsanov} with $\widehat{\mathbf{H}}\equiv0$ and $d=1$) without imposing the condition \eqref{zerospen} (see e.g. \cite[Example 1.22]{CE05}).
\end{remark}

\section{Some analytic results}\label{someanlyticresults}
The following proposition is used in the proof of Proposition \ref{feokokoe1} to show that $f g \in VMO_{loc}$ if $f,g \in VMO_{loc} \cap L_{loc}^{\infty}(\mathbb{R}^d)$.
\begin{proposition} \label{basicpropvm}
\begin{itemize}
\item[(i)]
If $f,g \in VMO \cap L^{\infty}(\mathbb{R}^d)$, then $fg \in VMO \cap L^{\infty}(\mathbb{R}^d)$.

\item[(ii)]
If $B$ is an open ball in $\mathbb{R}^d$ and $f,g \in VMO(B) \cap L^{\infty}(B)$, then $fg \in VMO(B) \cap L^{\infty}(B)$. In particular, there exists $\widetilde{f}, \widetilde{g} \in VMO \cap L^{\infty}(\mathbb{R}^d)$ such that $\widetilde{f}|_{B}=f$ and $\widetilde{g}|_B=g$ on $B$.

\item[(iii)]
If $f,g \in VMO_{loc} \cap L^{\infty}_{loc}(\mathbb{R}^d)$, then $fg \in VMO_{loc} \cap L^{\infty}_{loc}(\mathbb{R}^d)$.
\end{itemize}
\end{proposition}
\begin{proof}
(i) Let $f,g \in VMO \cap L^{\infty}(\mathbb{R}^d)$. Then, there exist positive continuous functions $\omega_1$ and $\omega_2$ on $[0, \infty)$ with $\omega_1(0)=\omega_2(0)=0$ such that for all $z \in \mathbb{R}^d$ and $R>0$
$$
\sup_{r<R}\, r^{-2d}\int_{B_r(z)} \int_{B_r(z)} |f(x)-f(y)| \,dx dy \leq  \omega_1(R), \; \;\;\;
\sup_{r<R}\, r^{-2d}\int_{B_r(z)} \int_{B_r(z)} |g(x)-g(y)| \,dx dy \leq  \omega_2(R).
$$
Now let $z \in \mathbb{R}^d$ and $R>0$. Take $r \in (0, R)$. Then, for any $x, y \in B_r(z)$ it holds
$$
|f(x)g(x) - f(y)g(y)| \leq \|f\|_{L^{\infty}(\mathbb{R}^d)} |g(x)-g(y)| +\|g\|_{L^{\infty}(\mathbb{R}^d)} |f(x)-f(y)|.
$$
Thus,
$$
r^{-2d}\int_{B_r(z)} \int_{B_r(z)} |f(x)g(x)-f(y)g(y)| \,dx dy \leq  \|f\|_{L^{\infty}(\mathbb{R}^d)}\omega_1(R) +\|g\|_{L^{\infty}(\mathbb{R}^d)}\omega_2(R).
$$
Taking the supremum over $z \in \mathbb{R}^d$ and $r \in (0, R)$, the assertion follows. \\
(ii)
Let $f, g \in VMO(B) \cap L^{\infty}(B)$. Then there exists $\bar{f}, \bar{g} \in VMO$ such that $\bar{f}(x) = f(x)$  and $\bar{g}(x) = g(x)$ for all $x \in B$. Let $M= \|f\|_{L^{\infty}(B)} + \|g\|_{L^{\infty}(B)}+1$ and $\eta \in C_0^{\infty}(\mathbb{R})$ be a function satisfying $\eta(x)=x$ for all $x \in [-M,M]$. Let $\widetilde{f} = \eta \circ \bar{f}$ and $\widetilde{g}=\eta \circ \bar{g}$. Then, $\widetilde{f}, \widetilde{g} \in L^{\infty}(\mathbb{R}^d)$ satisfy $\widetilde{f}(x) = f(x)$ and $\widetilde{g}(x) = g(x)$ for a.e. $x \in B$. Moreover, $\widetilde{f}, \widetilde{g} \in VMO$ by the mean value theorem. Therefore, the assertion follows from (i).\\
(iii) The assertion directly follows from (ii).
\end{proof}

\begin{proposition} \label{vmoprop}
\begin{itemize}
\item[(i)]
Let $B,U \subset \mathbb{R}^d$ be open balls in $\mathbb{R}^d$ with $\overline{B} \subset U$. Let $A=(a_{ij})_{1 \leq i,j \leq d}$ be a (possibly non-symmetric) matrix of functions on $U$ with $a_{ij} \in VMO(U) \cap L^{\infty}(U)$ for all $1 \leq i,j \leq d$ such that there exists a constant $\lambda_U >0$ such that
\begin{equation} \label{locellipt}
\lambda_U \| \xi \|^2 \leq \langle A(x) \xi, \xi \rangle, \quad \text{ for a.e. $x \in U$ and all\, $\xi \in \mathbb{R}^d$}.
\end{equation}
Then, there exists a matrix of measurable functions $\widetilde{A}^B=(\widetilde{a}^B_{ij})_{1 \leq i,j \leq d}$ on $\mathbb{R}^d$
such that $\widetilde{a}_{ij}^B \in VMO \cap L^{\infty}(\mathbb{R}^d), \, \widetilde{a}_{ij}^B|_B = a_{ij}$ on $B$ for all $1 \leq i,j \leq d$ and that
\begin{equation} \label{globellipti}
\lambda_U \|\xi\| \leq \langle \widetilde{A}^B(x) \xi, \xi \rangle, \quad \text{ for a.e. $x \in \mathbb{R}^d$ and all\, $\xi \in \mathbb{R}^d$}.
\end{equation}

\item[(ii)]
Let $A=(a_{ij})_{1 \leq i,j \leq d}$ be a (possibly non-symmetric) matrix of functions on $\mathbb{R}^d$ with $a_{ij} \in VMO_{loc} \cap L^{\infty}_{loc}(\mathbb{R}^d)$ for all $1 \leq i,j \leq d$ such that for each open ball $U$ there exists a constant $\lambda_U>0$ such that \eqref{locellipt} holds. Then, for each open ball $B$ there exists a matrix of functions $\widetilde{A}^B=(\widetilde{a}^B_{ij})_{1 \leq i,j \leq d}$ on $\mathbb{R}^d$
such that $\widetilde{a}^B_{ij} \in VMO \cap L^{\infty}(\mathbb{R}^d)$, $\widetilde{a}_{ij}^B|_B = a_{ij}$ on $B$ and that for some constant $\lambda>0$ it holds
\begin{equation} \label{ellipticequa}
\lambda \|\xi\|^2 \leq \langle \widetilde{A}^B(x) \xi, \xi \rangle \quad \text{ for a.e. $x \in \mathbb{R}^d$ and all\, $\xi \in \mathbb{R}^d$}.
\end{equation}
\end{itemize}
\end{proposition}
\begin{proof}
(i) Since $a_{ij} \in VMO(U)$, by the Proposition \ref{basicpropvm}(ii) there exists $\overline{a}_{ij} \in VMO \cap L^{\infty}(\mathbb{R}^d)$ such that $\overline{a}_{ij}|_{U}=a_{ij}$ on $U$. Set $\overline{A}=(\overline{a}_{ij})_{1\le i,j \le d}$.
Let $\chi \in C_0^{\infty}(\mathbb{R}^d)$, $\text{supp}(\chi)\subset U$, be such that $0 \leq \chi \leq 1$ on $\mathbb{R}^d$ and $\chi \equiv 1$ on $B$. Define $\widetilde{A}^B=(\widetilde{a}_{ij})_{1 \leq i,j \leq d}$ by
$$
\widetilde{A}^B = \chi \overline{A} + (1-\chi) \lambda_U id= \chi (\overline{A}-\lambda_U id) + \lambda_U id \quad \text{on $\mathbb{R}^d$}.
$$
Then, by the Proposition \ref{basicpropvm}(i) $\widetilde{a}_{ij}^B \in VMO \cap L^{\infty}(\mathbb{R}^d)$ for all $1 \leq i,j \leq d$. Moreover, \eqref{globellipti} follows.\\
(ii)
Let $B$ be a given open ball. Take an open ball $U$ with $\overline{B} \subset U$. Since $a_{ij}|_U \in VMO(U) \cap L^{\infty}(U)$ for all $1 \leq i,j \leq d$, the assertion now follows from (i), where \eqref{ellipticequa} holds with $\lambda=\lambda_U$.
\end{proof}

\begin{lemma} \label{stoneweier}
Assume {\bf (S1)}. Let $U$ be a bounded open subset of $\R^d$ and $T>0$. Set
\begin{eqnarray*}
&&\mathcal{S}:= \big \{h \in C_0^{\infty}(U \times (0,T)) : \text{there exists $N \in \N$ such that $h= \sum_{i=1}^N f_i g_i$,} \\
&& \qquad \qquad \qquad \qquad \text{where $f_i \in C_0^{\infty}(U)$, $g_i \in C_0^{\infty}((0,T))$ for all i=1,\dots, N }    \big \}.
\end{eqnarray*}
Then $C_0^{2}(U \times (0,T)) \subset \overline{\mathcal{S}}|_{C^2(\overline{U} \times [0,T])}$.
\end{lemma}
\begin{proof} {\bf Step 1}: Let $V$ be a bounded open set in $\R^d$ and $T_1, T_2 \in \R$ with $T_1<T_2$.
Define 
\begin{eqnarray*}
&&\mathcal{R}:= \big \{h \in C_0^{\infty}(V \times (T_1,T_2))\mid \text{there exists $N \in \N$ such that $h= \sum_{i=1}^N f_i g_i$,} \\
&& \qquad \qquad \qquad \qquad \text{where $f_i \in C_0^{\infty}(V)$, $g_i \in C_0^{\infty}((T_1,T_2))$ for all $i=1,\dots, N$} \big \} .
\end{eqnarray*}
We claim that 
\begin{equation} \label{step1fr}
C_0^{2}(V \times (T_1,T_2)) \subset \overline{\mathcal{R}}|_{C(\overline{V} \times [T_1,T_2])}.
\end{equation}
Note that $V \times (T_1,T_2)$ is a locally compact space and $\overline{\mathcal{R}}|_{C(\overline{V} \times [T_1,T_2])}$ is a closed subalgebra of $C_{\infty}(V \times (T_1,T_2))$. We can easily check that for each $(x,t) \in V \times (T_1,T_2)$, there exists $\widetilde{h} \in \mathcal{R}$ such that $\widetilde{h}(x,t) \neq 0$. For $(x,t), (y,s) \in V \times (T_1,T_2)$ and $(x,t) \neq (y,s)$, there exists $\widehat{h} \in \mathcal{R}$ such that $\widehat{h}(x,t)=1$ and $\widehat{h}(y,s)=0$. Therefore by \cite[Chapter V, 8.3 Corollary]{Conw},  we obtain $\overline{\mathcal{R}}|_{C(\overline{V} \times [T_1,T_2])}= C_{\infty}(V \times (T_1,T_2))$ (the continuous functions on $V \times (T_1,T_2)$ that vanish at infinity, i.e. given $\varepsilon>0$, there exists a compact set $K \subset V \times (T_1,T_2)$ such that $|f(x)|< \varepsilon $ for all $x \in V \times (T_1,T_2) \setminus K$), so that our claim \eqref{step1fr} holds. \\
\centerline{}
{\bf Step 2}: $C_0^{2}(U \times (0,T)) \subset \overline{\mathcal{S}}|_{C^2(\overline{U} \times [0,T])}$. \\
For $n \in \N$, let $\eta_{n}$ be a standard mollifier on $\R^d$ and $\theta_{n}$ be a standard mollifier on $\R$. Then $\xi_{n}:= \eta_{n} \theta_{n}$ is a standard mollifier on $\R^d \times \R$. Let $h \in C_0^{2}(U \times (0,T))$ be given. Then there exists a bounded open subset $V$ of $\R^d$ and $T_1, T_2 \in \R$ with $0<T_1<T_2$ such that
$$
\text{supp}(h) \subset V \times (T_1, T_2) \subset \overline{V} \times [T_1, T_2] \subset U \times (0,T).
$$
Take $N \in \N$ such that $f*\xi_N \in C_0^{\infty}(U \times (0,T))$ for all $f \in C_0^{\infty}(V \times (T_1,T_2))$.\\
Note that by \cite[Proposition 4.20]{BRE}, it holds
\begin{eqnarray*}
&&\partial_t(h * \xi_{\varepsilon}) = \partial_t  h * \xi_{\varepsilon}, \; \partial^2_t (h * \xi_{\varepsilon}) = \partial^2_t  h * \xi_{\varepsilon}, \; \partial_t \partial_i  (h * \xi_{\varepsilon}) = \partial_t \partial_i  h * \xi_{\varepsilon},  \\
&& \partial_i (h * \xi_{\varepsilon}) = \partial_i  h * \xi_{\varepsilon}, \,\;\; \partial_i \partial_j (h * \xi_{\varepsilon}) =  \partial_i \partial_j  h * \xi_{\varepsilon}\, \text{ for any $1 \leq i,j \leq d$.}
\end{eqnarray*}
Hence by \cite[Proposition 4.21]{BRE}, $\lim_{n \rightarrow \infty}h *\xi_{\varepsilon} = h$ in $C^2(\overline{U} \times [0,T])$. 
Thus given $\varepsilon>0$, there exists $n_{\varepsilon} \in \N$ with $n_{\varepsilon} \geq N$ such that
$$
\|h - h* \xi_{n_{\varepsilon}}\|_{C^2(\overline{U} \times [0,T])} < \frac{\varepsilon}{2}.
$$
Let $\mathcal{R}$ be as in Step 1. By \eqref{step1fr}, there exists $h_{\varepsilon} \in \mathcal{R} \subset  C_0^{\infty}(V \times (T_1,T_2))$ such that
$$
\|h-h_{\varepsilon}\|_{C(\overline{U} \times [0,T])} < \frac{\varepsilon}{2 \|\xi_{n_{\varepsilon}}\|_{C^2(\overline{U} \times [0,T])}}.
$$
Thus using \cite[Proposition 4.20]{BRE} and Young's inequality,
$$
\|h * \xi_{n_{\varepsilon}} - h_{\varepsilon}* \xi_{n_{\varepsilon}}\|_{C^2(\overline{U} \times [0,T])} \leq \|\xi_{n_{\varepsilon}}\|_{C^2(\overline{U} \times [0,T])}\|h-h_{\varepsilon}\|_{C(\overline{U} \times [0,T])}< \frac{\varepsilon}{2}.
$$
Therefore
$$
\|h-h_{\varepsilon}* \xi_{n_{\varepsilon}}\|_{C^2(\overline{U} \times [0,T])} <\varepsilon.
$$
Since $h_{\varepsilon}* \xi_{n_{\varepsilon}} \in \mathcal{S}$, we have $h \in  \overline{\mathcal{S}}|_{C^2(\overline{U} \times [0,T])}$, as desired. 
\end{proof}

\begin{lemma} \label{lem7.3}
Let $A=(a_{ij})_{1 \leq i,j \leq d}$ be a (possibly non-symmetric) matrix of measurable functions on $\mathbb{R}^d$ such that there exist constants $\lambda>0$ and $M>0$, such that for all $\xi = (\xi_1, \dots, \xi_d) \in \R^d$, and a.e. $x \in \mathbb{R}^d$ 
$$
\quad \sum_{i,j=1}^{d} a_{ij}(x) \xi_i \xi_j \geq \lambda\|\xi\|^2, \quad \max_{1 \leq i, j \leq d} | a_{ij}(x)   | \leq M.
$$
Let $r \in (1, d)$ and $p \in (d, \infty)$ and $U$ be an open ball in $\R^d$. Assume that $a_{ij} \in VMO$, $1 \leq i,j \leq d$. Let $\widetilde{\mathbf{B}} \in L^{d}(U, \mathbb{R}^d)$ if $d \geq 3$ and $\widetilde{\mathbf{B}}\in L^{p}(U, \mathbb{R}^d)$ if $d=2$, and $\widetilde{f} \in L^{r}(U)$, $\widetilde{\mathbf{F}} \in L^{\frac{dr}{d-r}}(U, \mathbb{R}^d)$. 
Consider the following variational equation, to be satisfied for a sufficiently regular weakly differentiable function $u$:
\begin{align} \label{varidenpd}
\int_{U} \langle A \nabla u, \nabla \varphi \rangle dx  + \int_{U} \langle \widetilde{\mathbf{B}}, \nabla u \rangle \varphi dx    = \int_{U} \widetilde{f} \varphi dx+ \int_{U} \langle\widetilde{\mathbf{F}}, \nabla \varphi \rangle dx 
\quad \text{ for all $\varphi \in C_0^{\infty}(U)$}.
\end{align}
Then, the following hold:
\begin{itemize}
\item[(i)]
There exists a unique $u \in H^{1,\frac{dr}{d-r}}_0(U)$ solving \eqref{varidenpd} and a constant $C_1>0$ independent of $\widetilde{f}$ and $\widetilde{\mathbf{F}}$ such that
\begin{equation*} \label{estim}
\| \nabla u \|_{L^{\frac{dr}{d-r}}(U)} \leq C_1 \big( \| \widetilde{f}\|_{L^{r}(U)} + \| \widetilde{\mathbf{F}}\|_{L^{\frac{dr}{d-r}}(U)} \big).
\end{equation*}

\item[(ii)]
Assume that $\widetilde{\mathbf{F}}=0$, $\widetilde{f} \in L^p(U)$, ${\rm div} A \in L^p(U, \mathbb{R}^d)$ and $\widetilde{B} \in L^p(U, \mathbb{R}^d)$ (for any $d\ge 2$). Then, there exists a unique $u \in H^{1,2}_0(U) \cap H^{2,p}(U)$ solving
\eqref{varidenpd} and it holds
$$
\| u \|_{H^{2,p}(U)} \leq C_2 \| \widetilde{f}\|_{L^{p}(U)},
$$
where $C_2>0$ is a constant independent of $\widetilde{f}$. 
\end{itemize}
\end{lemma}
\begin{proof}
(i) Using \cite[Theorem 9.15, Lemma 9.17]{Gilbarg}, there exists $\widetilde{u} \in H^{2,r}(U) \cap H^{1,r}_0(U)$ such that $-\Delta \widetilde{u} = \widetilde{f}$ in $U$ and that
$$
\|\widetilde{u}\|_{H^{2,r}(U)} \leq C_1 \|\widetilde{f}\|_{L^r(U)},
$$
where $C_1>0$ is a constant independent of $\widetilde{f}$. Using Sobolev's embedding, 
\begin{equation} \label{w1restim2}
\|\widetilde{u}\|_{H^{1, \frac{dr}{d-r}}(U)} \leq C_2 \|\widetilde{f}\|_{L^r(U)},
\end{equation}
where $C_2>0$ is a constant independent of $\widetilde{f}$. Let $\widetilde{\mathbf{G}}:= \nabla \widetilde{u}$. Then  $\widetilde{\mathbf{G}}\in L^{\frac{dr}{d-r}}(U,\R^d)$ and by \cite[Theorem 2.1(i)]{KK19} there exists a unique $u \in H^{1, \frac{dr}{d-r}}_0(U)$ such that
\begin{align*}
\int_{U} \langle A \nabla u, \nabla \varphi \rangle dx  + \int_{U} \langle \widetilde{\mathbf{B}}, \nabla u \rangle \varphi dx   &= \int_{U} \langle \widetilde{\mathbf{G}}, \nabla \varphi \rangle dx+ \int_{U} \langle \widetilde{\mathbf{F}}, \nabla \varphi \rangle dx\\
&=\int_{U} \widetilde{f} \varphi dx+ \int_{U} \langle\widetilde{\mathbf{F}}, \nabla \varphi \rangle dx \quad \text{ for all $\varphi \in C_0^{\infty}(U)$}.
\end{align*}
Indeed, to check  that the conditions of \cite[Theorem 2.1(i)]{KK19} are satisfied is straightforward for $d\ge 3$, and for $d=2$, we note that $\widetilde{\mathbf{B}} \in L^{2+\varepsilon}(U,\R^d)$ for any $\varepsilon>0$. Thus, choosing $\varepsilon>0$ small enough, we can achieve $\frac{2+\varepsilon}{2+\varepsilon-1}<\frac{2r}{2-r}$ for any given $r\in (1,2)$.
Moreover,  by \cite[Theorem 2.1(i)]{KK19} and \eqref{w1restim2}, we get
\begin{align*}
\| \nabla u \|_{L^{\frac{dr}{d-r}}(U)} &\leq C_3 \big ( \|\widetilde{\mathbf{G}}\|_{L^{\frac{dr}{d-r}}(U)}  + \|  \widetilde{\mathbf{F}}\|_{L^{\frac{dr}{d-r}}(U)} \big )\\
& \leq C_4  \big( \| \widetilde{f}\|_{L^r(U)} + \| \widetilde{\mathbf{F}} \|_{L^{\frac{dr}{d-r}}(U)} \big),
\end{align*}
where $C_3, C_4>0$ are constants independent of $\widetilde{f}$ and $\widetilde{\mathbf{F}}$, as desired.\\
\noindent
(ii) Since $\frac{3}{2d}-\frac{1}{2p} \in (\frac{1}{d}, 1)$, $\widetilde{f} \in L^p(U) \subset L^{\left( \frac{3}{2d}- \frac{1}{2p} \right)^{-1}}(U)$. By (i) and Morrey's inequality, there exists a unique $u \in H_0^{1,\left( \frac{1}{2d}- \frac{1}{2p} \right)^{-1}}(U) \cap C(\overline{U})$ such 
\eqref{varidenpd} holds and that
\begin{equation} \label{intestim}
\|u\|_{L^{\infty}(U)} \leq  C_5 \| \nabla u \|_{L^{\left( \frac{1}{2d}- \frac{1}{2p} \right)^{-1}}(U)} \leq C_6 \| \widetilde{f}\|_{L^{p}(U)},
\end{equation}
where $C_5>0$ is a constant independent of $\widetilde{f}$. Let $\widetilde{h}:= \langle {\rm div} A, \nabla u  \rangle    
-\langle \widetilde{\mathbf{B}}, \nabla u \rangle$. Since $\nabla u \in L^{\left( \frac{1}{2d}- \frac{1}{2p} \right)^{-1}}(U)$, ${\rm div} A \in L^p(U, \mathbb{R}^d)$ and $\widetilde{\mathbf{B}} \in L^{p}(U, \mathbb{R}^d)$, we get $\widetilde{h} \in L^{(\frac{1}{2d}+\frac{1}{2p})^{-1}}(U)$ and $\langle {\rm div} A, \nabla u \rangle  \in L^{(\frac{1}{2d}+\frac{1}{2p})^{-1}}(U)$. 
Let $\lambda_0>0$ be the constant independent of $\widetilde{f}$ as in \cite[Theorem 8]{DK11} applied with  $L = \sum_{i,j=1}^d a_{ij}\partial_i \partial_j$ there.
Let $\widetilde{g}:=\widetilde{f}+\lambda_0 u  +\widetilde{h}$. Then $\widetilde{g} \in L^{(\frac{1}{2d}+\frac{1}{2p})^{-1}}(U)$ and by \cite[Theorem 8]{DK11}, there exists a unique $w \in H^{2, \left( \frac{1}{2d}+\frac{1}{2p}  \right)^{-1}}(U) \cap H^{1,2}_0(U)$  such that
\begin{equation} \label{nondivou}
-\text{trace}(A \nabla ^2 w) +\lambda_0 w = \widetilde{g} \quad \text{ on $U$},
\end{equation}
and that by \eqref{intestim}
$$
\|w\|_{H^{2, (\frac{1}{2d}+\frac{1}{2p})^{-1}}(U)} \leq C_{7} \|\widetilde{g}\|_{L^{(\frac{1}{2d}+\frac{1}{2p})^{-1}}(U)} \leq C_{8} \|\widetilde{f}\|_{L^p(U)},
$$
where $C_{7}, C_{8}>0$ are constants independent of $\widetilde{f}$. Since $(\frac{1}{2d}+\frac{1}{2p})^{-1}>d$, using Sobolev's embedding, $\nabla w \in L^{\infty}(U)$ and 
\begin{equation} \label{bddgrades}
\| \nabla w\|_{L^{\infty}(U)} \leq C_{9} \|\widetilde{f}\|_{L^p(U)},
\end{equation}
where $C_{9}>0$ is a constant independent of $\widetilde{f}$. Now note that from \eqref{nondivou} and integration by parts
$$
\int_{U} \langle A \nabla w, \nabla \varphi \rangle dx  + \int_{U} \langle {\rm div} A, \nabla w \rangle \varphi dx +\lambda_0 \int_{U} w \varphi dx  = \int_{U} \widetilde{g}\varphi dx, \quad \text{ for all $\varphi \in C_0^{\infty}(U)$}.
$$
From \eqref{varidenpd} we get
$$
\int_{U} \langle A \nabla u, \nabla \varphi \rangle dx  + \int_{U} \langle {\rm div} A, \nabla u \rangle \varphi dx +\lambda_0 \int_{U} u \varphi dx  = \int_{U} \widetilde{g}\varphi dx, \quad \text{ for all $\varphi \in C_0^{\infty}(U)$}.
$$
Using a weak maximum principle (\cite[Theorem 2.1.8]{BKRS}, see \cite[Theorem 4]{T77} for the original result), $u = w$ on $U$. Thus, $\nabla u =\nabla w \in L^{\infty}(U, \mathbb{R}^d)$, hence $\phi \in L^p(U)$. Moreover, using \eqref{intestim}, \eqref{bddgrades} and applying \cite[Theorem 8]{DK11} to \eqref{nondivou} again, we get $u=w \in H^{2,p}(U)$ and
\begin{align*}
&\|u \|_{H^{2,p}(U)} \leq C_{10} \|\widetilde{g}\|_{L^p(U)}  \\
&\leq C_{10} \big(  \|\widetilde{f}\|_{L^p(U)}+  \lambda_0 \|u\|_{L^p(U)}  +(\|{\rm div} A \|_{L^p(U)} + \| \widetilde{B} \|_{L^p(U)}) \|\nabla u \|_{L^{\infty}(U)} \big) \\
& \leq
C_{11} \|\widetilde{f}\|_{L^p(U)},
\end{align*}
where $C_{10}, C_{11}>0$ are constants independent of $\widetilde{f}$, as desired.
\end{proof}

\begin{remark}
In the situation of Lemma \ref{lem7.3}(ii), one may directly use \cite[Theorem 2.2]{KK19} to obtain the same conclusion of Lemma \ref{lem7.3}(ii) if the conditions, $B \in L^p(U, \mathbb{R}^d)$ and ${\rm div} A \in L^p(U, \mathbb{R}^d)$ in Lemma \ref{lem7.3}(ii) are replaced by $B \in L^{p+\varepsilon}(U, \mathbb{R}^d)$ and ${\rm div} A \in L^{p+\varepsilon}(U, \mathbb{R}^d)$ for some $\varepsilon>0$. But for our main result, with optimal conditions on the coefficients, we will use Lemma \ref{lem7.3}(ii) (which is nonetheless derived on the basis of results from \cite{KK19}) with $p = d+1$ rather than \cite[Theorem 2.2]{KK19}. 
\end{remark}

\centerline{}

\begin{theorem} \label{mainregulthm}
Let $A=(a_{ij})_{1 \leq i,j \leq d}$ be a (possibly non-symmetric)  matrix of measurable functions on $\mathbb{R}^d$ such that there  exist constants $\lambda>0$ and $M>0$, such that for all $\xi = (\xi_1, \dots, \xi_d) \in \R^d$, and a.e. $x \in \mathbb{R}^d$ it holds
$$
\quad \sum_{i,j=1}^{d} a_{ij}(x) \xi_i \xi_j \geq \lambda\|\xi\|^2, \quad \max_{1 \leq i, j \leq d} | a_{ij}(x)   | \leq M.
$$
Let $p \in (d, \infty)$ and $V$ be an open ball in $\R^d$. Let $\widetilde{\mathbf{B}} \in L^p(V, \mathbb{R}^d)$, $\widetilde{c} \in L^\frac{pd}{p+d}(V)$ and $\widetilde{f} \in L^{\frac{pd}{p+d}}(V)$, $\widetilde{\mathbf{F}} \in L^{p}(V, \mathbb{R}^d)$. Let $u \in H^{1,2}(V)$ be such that
\begin{align} \label{undevaride}
\int_{V} \langle A \nabla u, \nabla \varphi \rangle dx  + \int_{V} \langle\widetilde{\mathbf{B}}, \nabla u \rangle \varphi dx + \int_{V} \widetilde{c} u \varphi dx   = \int_{V} \widetilde{f} \varphi dx + \int_{V} \langle \widetilde{\mathbf{F}},\nabla \varphi \rangle dx \quad \text{ for all $\varphi \in C_0^{\infty}(V)$}.
\end{align}
Let $B$ be an open ball in $\mathbb{R}^d$ with $\overline{B} \subset V$. Then the following hold:
\begin{itemize}

\item[(i)]
There exists a constant $C_1>0$ independent of $\widetilde{f}$ and $\widetilde{\mathbf{F}}$ such that
$$
\|u\|_{H^{1,2}(B)} \leq C_1 \big( \|u\|_{L^1(V)}  + \| \widetilde{f} \|_{L^{r}(V)} + \| \widetilde{\mathbf{F}}\|_{L^2(V)}  \big),
$$
where $r = \frac{2d}{2+d}$ if $d \geq 3$ and $r \in (1,2)$ can be arbitrarily chosen if $d=2$.

\item[(ii)]
There exist constants $C>0$ and $\gamma \in (0,1)$ independent of $\widetilde{f}$ and $\widetilde{\mathbf{F}}$ such that $u \in C^{0, \gamma}(\overline{B})$ and
$$
\|u\|_{C^{0, \gamma}(\overline{B})} \leq C \big (  \|u\|_{L^1(V)}  + \| \widetilde{f} \|_{L^{\frac{pd}{p+d}}(V)}  +  \|\widetilde{\mathbf{F}}\|_{L^p(V)}  \big ).
$$
\item[(iii)]
If $a_{ij} \in VMO$, $1 \leq i,j \leq d$, then $u \in H_{loc}^{1,p}(V)$ and there exists a constant $C_1>0$ independent of $\widetilde{f}$ such that
$$
\|u\|_{H^{1,p}(B)} \leq C_1 \big( \|u\|_{L^1(V)}  + \| \widetilde{f}\|_{L^{\frac{pd}{p+d}}(V)}   + \| \widetilde{\mathbf{F}}\|_{L^p(V)}   \big ),
$$
where $C_1>0$ is a constant independent of $\widetilde{f}$ and $\widetilde{\mathbf{F}}$.

\item[(iv)]
If $a_{ij} \in VMO$, $1 \leq i,j \leq d$, $\widetilde{\mathbf{F}}=0$, ${\rm div} A \in L^p(V, \mathbb{R}^d)$ and $\widetilde{c} \in L^p(V)$,  then $u \in H_{loc}^{2,p}(V)$ and there exists a constant $C_2>0$ independent of $f$ such that
$$
\|u\|_{H^{2,p}(B)} \leq C_2 \big ( \|u\|_{L^1(V)}  + \| \widetilde{f}\|_{L^p(V)}     \big ),
$$
where $C_2>0$ is a constant independent of $\widetilde{f}$.
\end{itemize}
\end{theorem}
\begin{proof}
(i) The assertion directly follows from a straightforward modification of \cite[Theorem 1.7.4]{BKRS}, since for the proof of \cite[Theorem 1.7.4]{BKRS}, the assumption that $A$ is symmetric and that $a_{ij} \in VMO(U)$ is not needed to obtain the $H^{1,2}$-estimate.\\
(ii) Let $U$ be an open ball  in $\mathbb{R}^d$ such that $\overline{B} \subset U \subset \overline{U} \subset V$.  As in the proof of Lemma \ref{lem7.3}(i), there exists $\widetilde{u} \in H^{2,\frac{pd}{p+d}}(U) \cap H^{1,\frac{pd}{p+d}}_0(U)$ such that $-\Delta \widetilde{u} = \widetilde{f}$ in $U$. Thus for $\widetilde{\mathbf{G}}:=\nabla \widetilde{u}$\, it holds $\widetilde{\mathbf{G}}\in L^p(U, \mathbb{R}^d)$,
$$
\int_{U}\langle\widetilde{\mathbf{G}} , \nabla \varphi \rangle dx = \int_{U} \widetilde{f} \varphi dx, \quad \text{ for all $\varphi \in C_0^{\infty}(U)$},
$$
and 
$$
\|\widetilde{\mathbf{G}} \|_{L^p(U)} \leq \widetilde{C}_1 \| \widetilde{f} \|_{L^{\frac{pd}{p+d}}(U)},
$$
where $\widetilde{C}_1>0$ is a constant independent of$\widetilde{f}$. Using the above estimate for $\widetilde{\mathbf{G}}$, (i), and the H\"{o}lder regularity result (\cite[Théorème 7.2]{St65}, cf. the proof of \cite[Theorem 2]{LT19de} to see how \cite[Théorème 7.2]{St65} is applied), we get $u \in C^{0, \gamma}(\overline{B})$ and
\begin{align*}
\|u\|_{C^{0, \gamma}(\overline{B})} &\leq  \widetilde{C}_2 \big(  \|u\|_{L^1(U)}   +  \| \widetilde{\mathbf{G}}+ \widetilde{\mathbf{F}}\|_{L^p(U)} \big) \\
 &\leq C \big (  \|u\|_{L^1(V)}  + \|\widetilde{f}\|_{L^{\frac{pd}{p+d}}(V)}  +  \|\widetilde{\mathbf{F}}\|_{L^p(V)}  \big ),
\end{align*}
where $C_1, C>0$ are constants independent of $\widetilde{f}$ and $\widetilde{\mathbf{F}}$.\\
(iii) Let $U$ be an open ball  in $\mathbb{R}^d$ such that $\overline{B} \subset U \subset \overline{U} \subset V$. Then, applying (ii) for the open ball $U$, we obtain that $u \in C(\overline{U}) \subset L^{\infty}(U)$ and that
$$
\|u\|_{C^{0, \gamma}(\overline{U})} \leq C \Big (  \|u\|_{L^1(V)}  + \| \widetilde{f} \|_{L^{\frac{pd}{p+d}}(V)}  +  \|\widetilde{\mathbf{F}}\|_{L^p(V)}  \Big ).
$$
Let $\zeta \in C_0^{\infty}(U)$ be a fixed with $\zeta=1$ on $B$ and let $\phi \in C_0^{\infty}(U)$ be arbitrary.
Then, 
\begin{align*}
&\int_{U} \langle A \nabla u, \nabla (\zeta \phi) \rangle dx = \int_{U} \langle A \nabla u, \nabla \phi  \rangle \zeta dx +  \int_{U} \langle A \nabla u, \nabla \zeta  \rangle \phi dx \\
& =\int_{U} \langle A \nabla (u \zeta), \nabla \phi  \rangle dx  - \int_{U} \langle A \nabla \zeta, \nabla \phi \rangle u dx  +  \int_{U} \langle A \nabla u, \nabla \zeta  \rangle \phi dx.
\end{align*}
and
$$
\int_{U} \langle \widetilde{\mathbf{B}}, \nabla u \rangle \zeta \phi dx = \int_{U} \langle \widetilde{\mathbf{B}}, \nabla (u \zeta ) \rangle \phi dx - \int_{U} \langle \widetilde{\mathbf{B}}, \nabla \zeta \rangle u \phi dx.
$$
Therefore,  applying  \eqref{undevaride}  to $\varphi=\zeta \phi$, we obtain that
\begin{align} \label{ourpdere}
\int_{U} \langle A \nabla (\zeta u), \nabla \phi  \rangle dx+ \int_{U} \langle \widetilde{\mathbf{B}}, \nabla (\zeta u) \rangle \phi dx 
= \int_{U} \widehat{f} \phi dx +  \int_{U} \langle \widehat{\mathbf{F}}, \nabla \phi \rangle dx,
\end{align}
where $\widehat{f} = -\widetilde{c} \zeta u - \langle A \nabla u, \nabla \zeta  \rangle+\langle\widetilde{\mathbf{B}}, \nabla \zeta \rangle u + \widetilde{f} \zeta +\langle \widetilde{\mathbf{F}}, \nabla \zeta \rangle \in L^{\frac{pd}{p+d} \wedge 2}(U)$, \; $\widehat{\mathbf{F}} = uA \nabla \zeta +\zeta \widetilde{\mathbf{F}}\in L^p(U)$. \\
\noindent
If $\frac{1}{2} \leq \frac{1}{p}+\frac{1}{d}$, then $\widehat{f} \in  L^{\frac{pd}{p+d}}(U)$. Thus, by Lemma \ref{lem7.3}(i), we obtain that $\zeta u \in H^{1,p}_0(U)$ and by (i) and (ii)
\begin{align*}
\|\zeta u \|_{H^{1,p}(U)}  &\leq   C_3 \Big ( \| \widehat{f}\|_{L^{\frac{pd}{p+d}}(U)} + \| \widehat{\mathbf{F}} \|_{L^{p}(U)} \Big) \\
&\leq  C_4 \Big ( \|u\|_{L^1(V)}  + \| \widetilde{f}\|_{L^{\frac{pd}{p+d}}(V)}   + \| \widetilde{\mathbf{F}} \|_{L^p(V)}   \Big).
\end{align*}
Now consider the case where $\frac{1}{2} >\frac{1}{p}+\frac{1}{d}$. Then $d\ge 3$, and there exists a unique $k \in \mathbb{N}\cup \{ 0 \}$ such that 
\begin{equation*} 
\frac{1}{p}+\frac{1}{d}\in \Big [\frac{1}{2}-\frac{k+1}{d},\frac{1}{2}-\frac{k}{d}\Big ), \quad 
\text{hence}  \quad \frac{1}{p}\in \Big [\frac{1}{2}-\frac{k+2}{d},\frac{1}{2}-\frac{k+1}{d}\Big ).
\end{equation*}
Thus, $\widetilde{\mathbf{F}}\in L^p(U) \subset L^{(\frac{1}{2}-\frac{k+1}{d})^{-1}}(U)$ and $\widetilde{f} \in L^{2}(U)$ with $2\in (1,d)$. We claim that for each $m=1, \ldots, k+1$, $u \in H_{loc}^{1, \left( \frac{1}{2}-\frac{m}{d}  \right)^{-1}}(V)$ and for each open ball $B$ with $\overline{B} \subset V$,
$$
\|u \|_{ H^{1, \big ( \frac{1}{2}-\frac{m}{d}  \big)^{-1}}(B)} \leq   C_5 \Big( \|u\|_{L^1(V)}  + \| \widetilde{f}\|_{L^{\frac{pd}{p+d}}(V)}   + \| \widetilde{\mathbf{F}} \|_{L^p(V)}   \Big),
$$
where $C_5>0$ (in the following the notation for the constants $C_i$ should be checked on consistency) is a constant independent of $\widetilde{f}$ and $\widetilde{\mathbf{F}}$.
To show the claim, we use induction. Since $\frac{1}{p}<\frac{1}{2}-\frac{1}{d}$, $\widehat{\mathbf{F}}  \in L^p(U) \subset L^{\left(\frac{1}{2}-\frac{1}{d}\right)^{-1}}(U)$ and $\widehat{f} \in L^{2}(U)$. Applying Theorem \ref{mainregulthm}(ii) and Lemma \ref{lem7.3}(ii) to \eqref{ourpdere}, we obtain $\zeta u \in H^{1, \left( \frac{1}{2}-\frac{1}{d}  \right)^{-1}}(U)$ and
\begin{align*}
\|u \|_{H^{1, \big( \frac{1}{2}-\frac{1}{d}  \big)^{-1}}(B)} &\leq \|\zeta u \|_{H^{1, \big( \frac{1}{2}-\frac{1}{d}  \big)^{-1}}(U)}  \leq  C_5 \Big(  \|  \widehat{f}\|_{L^{2}(U)}     +      \| \widehat{\mathbf{F}}\|_{L^{\left(\frac{1}{2}-\frac{1}{d}\right)^{-1}}(U)} \Big)  \\
&\leq  C_6 \Big( \|u\|_{L^1(V)}  + \| \widetilde{f}\|_{L^{\frac{pd}{p+d}}(V)}   + \| \widetilde{\mathbf{F}}\|_{L^p(V)}   \Big).
\end{align*}
Since $B$ is arbitrary with $\overline{B} \subset V$, the claim holds for $m=1$. Now assume that the claim holds for some $m \in \{1,2, \ldots, k\}$. Note that $\frac{1}{2}-\frac{m}{d}>\frac{1}{p}+\frac{1}{d}$, hence $\frac{1}{2}-\frac{m+1}{d}>\frac{1}{p}$. Thus, 
$\widehat{\mathbf{F}}  \in L^p(U) \subset L^{\left( \frac{1}{2} - \frac{m+1}{d}   \right)^{-1}}(U)$ and $\widehat{f} \in L^{\left(  \frac{1}{2}-\frac{m}{d} \right)^{-1}}(U)$. Applying Theorem \ref{mainregulthm}(ii)  and Lemma \ref{lem7.3}(ii) to \eqref{ourpdere}, we obtain $\zeta u \in H^{1, \left( \frac{1}{2}-\frac{m+1}{d}  \right)^{-1}}(U)$ and
\begin{align*}
\|u \|_{H^{1, \left( \frac{1}{2}-\frac{m+1}{d}  \right)^{-1}}(B)} &\leq \|\zeta u \|_{H^{1, \left( \frac{1}{2}-\frac{m+1}{d}  \right)^{-1}}(U)}  \leq  C_7 \big(  \|  \widehat{f}\|_{L^{\left(\frac{1}{2}-\frac{m}{d}\right)^{-1}}(U)}    +      \| \widehat{\mathbf{F}}  \|_{L^{\left(\frac{1}{2}-\frac{m+1}{d}\right)^{-1}}(U)} \big)  \\
&\leq  C_8 \big( \|u\|_{L^1(V)}  + \| \widetilde{f}\|_{L^{\frac{pd}{p+d}}(V)}   + \|\widetilde{\mathbf{F}}\|_{L^p(V)}   \big).
\end{align*}
Since $B$ is arbitrary with $\overline{B} \subset V$, $u \in H_{loc}^{1, \left( \frac{1}{2}-\frac{m+1}{d}  \right)^{-1}}(V)$. Thus, the claim is shown. Therefore, we then obtain that $u \in H_{loc}^{1, \left( \frac{1}{2}-\frac{k+1}{d}  \right)^{-1}}(V)\subset
H_{loc}^{1, \left( \frac{1}{p}+\frac{1}{d} \right)^{-1}}(V)$ and
$$
\|u\|_{H^{1, \left( \frac{1}{p}+\frac{1}{d} \right)^{-1}}(U)} \leq  C_9 \big( \|u\|_{L^1(V)}  + \| \widetilde{f}\|_{L^{\frac{pd}{p+d}}(V)}   + \| \widetilde{\mathbf{F}}\|_{L^p(V)}   \big).
$$
Finally applying Theorem \ref{mainregulthm}(ii) and Lemma \ref{lem7.3}(ii) to \eqref{ourpdere} with $\widehat{f} \in L^{\left( \frac{1}{p}+\frac{1}{d} \right)^{-1}}(U)$ and $\widehat{\mathbf{F}} \in L^p(U, \mathbb{R}^d)$ the assertion follows.\\
(iv) Note that by Theorem \ref{mainregulthm}(iii), $u \in H_{loc}^{1,p}(V)$ and
$$
\|u\|_{H^{1,p}(B)} \leq C_{10} \big( \|u\|_{L^1(V)}  + \| \widetilde{f}\|_{L^{\frac{pd}{p+d}}(V)}      \big),
$$
where $C_{10}>0$ is a constant independent of $\widetilde{f}$. 
Using integration by parts in \eqref{ourpdere}, we obtain that
\begin{align*} \label{ourpdere2}
\int_{U} \langle A \nabla (\zeta u), \nabla \phi  \rangle dx+ \int_{U} \langle \widetilde{\mathbf{B}}, \nabla (\zeta u) \rangle \phi dx 
= \int_{U} \widehat{g} \phi dx,
\end{align*}
where 
$\widehat{g} = -\widetilde{c} \zeta u -\langle A \nabla u, \nabla \zeta  \rangle+\langle \widetilde{\mathbf{B}}, \nabla \zeta \rangle u + \widetilde{f} \zeta -\langle \nabla u,  A \nabla \zeta \rangle- u \cdot  \text{div}(A \nabla \zeta)  \in L^{p}(U)$.
Thus, applying Lemma \ref{lem7.3}(ii) and Theorem \ref{mainregulthm}(iii) 
to \eqref{ourpdere}, we obtain that $\zeta u \in H^{1,2}_0(U) \cap H^{2,p}(U)$ and that
\begin{align*}
\| \zeta u \|_{H^{2,p}(U)} \leq C_{11} \| \widehat{g}\|_{L^{p}(U)} \leq C_{12} \big( \|u\|_{L^1(V)}  + \| \widetilde{f}\|_{L^p(V)}     \big),
\end{align*}
where $C_{11}$, $C_{12}>0$ are constants independent of $\widetilde{f}$.
\end{proof}

\begin{theorem} \label{weoifo9iwjeof}
Let $V \times (0,T)$ be a bounded open subset in $\R^d\times \R$ with $T>0$.  Let $A=(a_{ij})_{1 \leq i,j \leq d}$ be a matrix of measurable functions on $V$ such that there  exist constants $\lambda>0$ and $M>0$, such that for all $\xi = (\xi_1, \dots, \xi_d) \in \R^d$, and a.e. $x \in V$ it holds
$$
\quad \sum_{i,j=1}^{d} a_{ij}(x) \xi_i \xi_j \geq \lambda\|\xi\|^2, \qquad \max_{1 \leq i, j \leq d} | a_{ij}(x)   | \leq M.
$$
Let $\widetilde{\mathbf{B}} \in L^{\widetilde{p}}(U,\mathbb{R}^d)$ and $\psi \in L^{\widetilde{q}}(U)$ with $\widetilde{p} \in (d, \infty)$ and $\widetilde{q} \in [2 \vee \frac{\widetilde{p}}{2},  \infty)$. Assume that there exists a constant $c_0>0$ such that $c_0 \leq \psi$ a.e. on $U$. Let $u \in H^{1,2}(V \times (0,T) )\cap L^\infty(V \times (0,T))$ satisfy 
\begin{eqnarray*} 
\;\; \iint_{U \times (0,T)} (u \partial_t\varphi) \psi dx dt = \iint_{U \times (0,T)}   \big \langle A\nabla u, \nabla \varphi  \big \rangle +\langle \widetilde{\mathbf{B}}, \nabla u \rangle \varphi \,dx dt, \quad \text{ for all }\varphi \in C_0^{\infty}(U \times (0,T)). \quad \label{baseq}
\end{eqnarray*}
Let $U$, $V$ be open subsets of $\mathbb{R}^d$ with $\overline{U} \subset V$ and $0<\tau_3<\tau_1<\tau_2<\tau_4<T$, i.e. $[\tau_1, \tau_2]\subset (\tau_3,\tau_4) \subset (0, T)$. Then, the following estimate holds:
\begin{equation*}\label{supest}
\|u\|_{L^{\infty}(U \times (\tau_1, \tau_2))} \leq C_1 
\|u\|_{L^{\frac{2\widetilde{p}}{\widetilde{p}-2},2}(V \times (\tau_3,\tau_4))},
\end{equation*}
where $C_1>0$ is a constant depending only on $r$, $\lambda$, $M$ and $\|\widetilde{\mathbf{B}}\|_{L^{\widetilde{p}}(U)}$ and $U$.
\end{theorem}
\begin{proof}
For each $\bar{x} \in \overline{U}$, denote by $R_{\bar{x}}(r)$ the open cube in $\R^d$ of edge length $r>0$ centered at $\bar{x}$. Let $\tau_2^*:=\frac{\tau_2+T}{2}$ and take $r \in (0,\frac{\sqrt{\tau_1}}{3})$ satisfying that $R_{\bar{x}}(3r)\subset V$   for all  $\bar{x}\in \overline{U}$.
Then for all $(\bar{x}, \bar{t}) \in \overline{U} \times [\tau_1, \tau_2^*]$, we have $R_{\bar{x}}(3r) \times (\bar{t}-(3r)^2, \bar{t})\subset V \times (0, T)$ and $\{R_{\bar{x}}(r)\times  (\bar{t}-r^2, \bar{t})\,:\, (\bar{x}, \bar{t}) \in \overline{U} \times [\tau_1, \tau_2^*]\}$ is an open cover of $\overline{U} \times [\tau_1, \tau_2]$.
Using the compactness of $\overline{U} \times [\tau_1, \tau_2]$, there exist $(x_i, t_i) \in  \overline{U} \times [\tau_1, \tau^{*}_2]$, $i=1, \dots, N$, such that
\[
\overline{U} \times [\tau_1, \tau_2] \subset \bigcup_{i=1}^{N} R_{x_i}(r) \times (t_i-r^2, t_i). 
\]
By \cite[Theorem 1]{LT19de},
\begin{eqnarray*}
 \sup_{\overline{U} \times [\tau_1, \tau_2]} |u| &\leq&  \max_{i=1,\dots,N}\, \sup_{R_{x_i}(r) \times (t_i-r^2, t_i)} |u| \\
&\leq&  \max_{i=1,\dots, N}\, c_i\|u\|_{L^{\frac{2\widetilde{p}}{\widetilde{p}-2},2}\big(R_{x_i}(2r) \times (t_i-(2r)^2, t_i)\big)} \\
&\leq& \underbrace{(\max_{i=1,\dots,N}\, c_i)}_{=:C_1} \,\|u\|_{L^{\frac{2\widetilde{p}}{\widetilde{p}-2},2}\big(V \times (\tau_3, \tau_4)\big)},
\end{eqnarray*}
where $c_i>0$, $1 \leq i \leq N$, are constants which are independent of $u$.
\end{proof}

\centerline{}
\centerline{}
\centerline{}
Haesung Lee\\
Department of Mathematics and Big Data Science,  \\
Kumoh National Institute of Technology, \\
Gumi, Gyeongsangbuk-do 39177, South Korea, \\
E-mail: fthslt@kumoh.ac.kr \\ \\ 
Gerald Trutnau\\
Department of Mathematical Sciences and \\
Research Institute of Mathematics of Seoul National University,\\
1 Gwanak-Ro, Gwanak-Gu,
Seoul 08826, South Korea,  \\
E-mail: trutnau@snu.ac.kr
\end{document}